\documentclass[11pt,a4paper]{amsart}

\usepackage{amssymb,amsmath,epsfig,graphics,mathrsfs}
\usepackage{graphicx}

\usepackage{fancyhdr}
\usepackage{hyperref}
\hypersetup{
 colorlinks   = true,
 urlcolor     = blue,
 linkcolor    = blue,
 citecolor   = red ,
 bookmarksopen=true
}

\theoremstyle{plain}

\newtheorem{theorem}{Theorem}[section]
\newtheorem{lemma}[theorem]{Lemma}
\newtheorem{prop}[theorem]{Proposition}

\newtheorem{definition}[theorem]{Definition}

\newtheorem{remark}[theorem]{Remark}
\theoremstyle{definition}
\theoremstyle{remark}
\numberwithin{equation}{section} \numberwithin{figure}{section}

\def \de {\partial}

\def \phi {\varphi}

\def \R {\mathbb{R}}
\def \l {\lambda}

\def \LL {\mathscr L_a}

\def \G{\Gamma}
\newcommand{\Ba}{\mathscr B_z^{(a)}}
\newcommand{\paa}{z^a \de_z}
\def \vf{\varphi}
\def \ve{\varepsilon}

\newcommand{\p}{\partial}
\newcommand{\Rnn}{\mathbb R^{n+1}}
\newcommand{\Rnp}{\mathbb R^{n+1}_+}
\newcommand{\Rn}{\R^n}
\newcommand{\la}{\lambda}

\newcommand{\gga}{\G_{2}^{(a)}}
\newcommand{\ga}{\G^{(a)}}

\newcommand{\Bz}{\mathscr B_\zeta^{(a)}}
\newcommand{\paz}{\zeta^a \de_\zeta}
\newcommand{\Pa}{\mathscr P^{(a)}_t}

\begin{document}

\title{Two classical properties of the Bessel quotient $I_{\nu+1}/I_\nu$ and their implications in pde's}


\author{Nicola Garofalo}
\address{Dipartimento di Ingegneria Civile, Edile e Ambientale (DICEA) \\ Universit\`a di Padova\\ 35131 Padova, ITALY}

\email[Nicola
Garofalo]{nicola.garofalo@unipd.it}

\thanks{The author was supported in part by a Progetto SID (Investimento Strategico di Dipartimento) ``Non-local operators in geometry and in free boundary problems, and their connection with the applied sciences", University of Padova, 2017.}

\dedicatory{In ricordo di mio padre} 

\date{}

\begin{abstract} 
Two elementary and classical results about the Bessel quotient $y_\nu = \frac{I_{\nu+1}}{I_\nu}$ state that on the half-line $(0,\infty)$ one has for $\nu\ge -1/2$:
\begin{itemize}
\item[(i)] $0 < y_\nu< 1$;
\item[(ii)] $y_\nu$ is strictly increasing.
\end{itemize}
In this paper we show that (i) and (ii) have some nontrivial and interesting applications to pde's. As a consequence of them, we establish some sharp new results for a class of degenerate partial differential equations of parabolic type in $\Rnp\times (0,\infty)$ which arise in connection with the analysis of the fractional heat operator $(\p_t - \Delta)^s$ in $\Rn\times (0,\infty)$, see Theorems \ref{T:harnackbessel0}, \ref{T:LYext0}, \ref{T:struwe0} and \ref{T:poon0} below.
\end{abstract}

\maketitle

\tableofcontents

\section{Introduction}\label{intro}

A random variable $Y$ taking values from the nonnegative integers is called a Bessel random variable, or Bessel distribution, with parameters $\nu>-1$ and $z>0$ if 
\[
\operatorname{Pr}(Y = k) = \frac{1}{I_\nu(z) \G(k) \G(k+\nu+1)} \left(\frac z2\right)^{2k+\nu},\ \ \ \ \ \ k\in \mathbb N\cup\{0\},
\]
where $I_\nu$ is the modified Bessel function of the first kind, see (5.6) in \cite{PY} and also (1.1) in \cite{YK}. The function $I_\nu$ is not very stable and it is often useful to consider the much more stable Bessel quotient $y_\nu = \frac{I_{\nu+1}}{I_\nu}$. On the real half-line $z\ge 0$ these two functions are connected by the equation
\[
I_\nu(z) = \frac{1}{\G(\nu+1)} \left(\frac z2\right)^\nu \exp\left(\int_0^z y_\nu(t) dt\right),
\]
see Proposition \ref{P:conn} below. 

The Bessel quotient $y_\nu$ plays an important role in a variety of problems from the applied sciences. For instance, it enters in the von Mises-Fisher distribution as the logarithmic derivative of the reciprocal of the norming constant. For $n\ge 2$ let $\mathbb S^{n-1}$ be the unit sphere in $\Rn$ and indicate with $\sigma_{n-1} = 2 \pi^{n/2}/\G(n/2)$ its $(n-1)$-dimensional measure. Let $d\sigma$ indicate the normalized measure on $\mathbb S^{n-1}$, so that $\int_{\mathbb S^{n-1}} d\sigma(x) = 1$. A random vector $x\in \mathbb S^{n-1}$ has the $(n-1)$-dimensional \emph{von Mises-Fisher}  distribution $M_n(\omega,z)$ if its probability density function with respect to the uniform distribution is
\[
f(x;\omega,z) = \left(\frac z2\right)^{\frac n2 - 1} \frac{1}{\G(n/2) I_{\frac n2 - 1}(z)} \exp \left\{z<\omega,x>\right\}.
\]
The parameters $\omega\in \mathbb S^{n-1}$ and $z\ge 0$ are respectively called the \emph{mean direction} and the \emph{concentration parameter} of the distribution. 
For $n=2$, $M_2(\omega,z)$ is the distribution on the circle introduced in 1918 by R. von Mises in \cite{vonmises} to study the
deviations of measured atomic weights from integral values. When $n=3$, $M_3(\omega,z)$ are called the Fisher distributions since they were systematically studied in 1953 by R. Fisher, who used them to investigate statistical problems in paleomagnetism, see \cite{Fisher}. But Fisher distributions first appeared in Physics in the 1905 work of P. Langevin \cite{La}, where he showed that in a collection of weakly interacting dipoles of moments $\bf{m}$ subject to an external electric field, the directions of the dipoles $\frac{\bf{m}}{|\bf{m}|}$ have a Fisher distribution. For an interesting account on the von Mises-Fisher distributions we refer the reader to the book of Mardia and Jupp \cite{Ma}. Also, the paper by Schou \cite{Schou} contains various interesting statistical results about the concentration parameter of the distribution.  

Using Cavalieri's Principle and the Poisson representation in \eqref{prepInu} below, it is easy to recognize that for $z>0$
\[
\int_{\mathbb S^{n-1}} \exp \left\{z<\omega,x>\right\} d\sigma(x) =  \G(n/2) \left(\frac z2\right)^{1-\frac n2} I_{\frac n2 - 1}(z),
\]
and thus it is clear that for any $\omega\in \mathbb S^{n-1}$ and $z\ge 0$
\[
\int_{\mathbb S^{n-1}} f(x;\omega,z) d\sigma(x) = 1.
\]
In view of this latter identity the function of $z$,
\[
a_n(z) = \left(\frac z2\right)^{\frac n2 - 1} \frac{1}{\G(n/2) I_{\frac n2 - 1}(z)},
\]
is called the \emph{norming constant} of the distribution.
From formula \eqref{logder2} below we see that
\[
\frac{d}{dz} \log a_n^{-1}(z) = \frac{d}{dz} \log\left(\frac{I_{\frac n2 - 1}(z)}{z^{\frac n2 - 1}}\right) = y_{\frac n2 -1}(z).
\]
This formula underscores the key role of the Bessel quotient $y_{n/2 -1}$ in the von Mises-Fisher distribution $M_n(\omega,z)$.

In 1965 Raj Pal Soni established the following elementary, yet quite important, inequality concerning $y_\nu$, see \cite{So} and also Proposition \ref{P:soni} below: for every $z>0$ one has
\begin{equation}\label{soni0}
y_\nu(z) < 1,\ \ \ \ \ \ \ \ \ \ \ \ \ \ \ \ \ \ \nu > - 1/2.
\end{equation}
Later, Nasell observed in \cite{Na} that \eqref{soni0} is also true when $\nu = -1/2$, see \eqref{Ihalf2} below. The inequality \eqref{soni0} fails for large enough values of $z$ when $-1<\nu<-1/2$, see the Appendix in this paper, and in particular Proposition \ref{P:asy}
 and Remark \ref{R:asy}. 
  
 Another stronger property of the Bessel quotient is the following, see Proposition \ref{P:monotonicityy} below: 
\begin{itemize}
\item[(i)] when $\nu \ge - 1/2$ the function $y_\nu$ strictly increases on $(0,\infty)$ from $y_\nu(0) = 0$ to its asymptotic value $y_\nu(\infty) = 1$; 
\item[(ii)] if instead $-1<\nu < - 1/2$, then $y_\nu$ first increases to its absolute maximum $>1$, and then it becomes strictly decreasing to its asymptotic value $y_\nu(\infty) = 1$. 
\end{itemize}
It is clear that the strict monotonicity of $y_\nu$ in (i), combined with $y_\nu(\infty) = 1$, implies the inequality \eqref{soni0}.

In this paper we show that such two elementary and classical results \eqref{soni0} and (i) about the Bessel quotient $y_\nu$ have some nontrivial and interesting applications to pde's. As a consequence of them, we establish some sharp new results for a class of degenerate partial differential equations of parabolic type in $\Rnp\times (0,\infty)$ which arise in connection with the analysis of the fractional heat operator $(\p_t - \Delta)^s$ in $\Rn\times (0,\infty)$, see Theorems \ref{T:harnackbessel0}, \ref{T:LYext0}, \ref{T:struwe0} and \ref{T:poon0} below.

To set the stage, consider the Bessel process on the half-line $z>0$,
\begin{equation}\label{bes}
\Ba u = z^{-a} \frac{d}{dz} \left(z^a \frac{d}{dz} u\right) = u_{zz} + \frac az u_z.
\end{equation}
This operator has a fractal dimension given by the number $a+1$. Since we are interested in a positive dimension, we assume henceforth that $a>-1$. 

The Bessel process plays an ubiquitous role in many branches of pure and applied sciences. It is well known that if we consider the Laplace operator in $\R^n$, then $\Delta$ acts on functions $u(x) = f(z)$, which depend only on the distance to the origin $z = |x|$, as follows 
\[
\Delta u(x) = \mathscr B^{(n-1)}_z f(z).
\]
More in general, if in $\Rn$ we consider a function with cylindrical symmetry $u(x_1,x_2,...,x_{n}) = f(x_1,z)$, where $z = \sqrt{x_2^2 + ... + x_{n}^2}$, then letting $x_1 = x$, we have
\begin{equation}\label{ms}
\Delta u = \frac{\p^2 f}{\p x^2} + \frac{\p^2 f}{\p z^2} + \frac{n-2}{z} \frac{\p f}{\p z} = \frac{\p^2 f}{\p x^2} + \mathscr B^{(n-2)}_z  f.
\end{equation}
This observation was one of the motivating elements in the 1965 seminal paper of Muckenhoupt and Stein \cite{MS} (see equation (1.2) in their paper and the subsequent discussion). The operator in the right-hand side of \eqref{ms} also arose in Molchanov's 1967 paper \cite{Mo} on the Martin boundary for invariant processes on solvable groups. He focused on the group $\mathbb G = \begin{pmatrix} z & y \\ 0 & 1\end{pmatrix}$ of affine transformations on the line, and on the following subclass 
\begin{equation}\label{mol}
L_\nu = \p_{yy} + \p_{zz} + \frac{2\nu +1}z \p_z = \p_{yy} + \mathscr B^{(2\nu+1)}_z 
\end{equation}
of that of all left-invariant second-order elliptic operators with respect to such group action,
see formulas (2) and (3) in \cite{Mo}. In 1969 the work of Molchanov and Ostrovskii \cite{MO} introduced in probability the idea of traces of Bessel processes.

In 2007 Caffarelli and Silvestre's celebrated extension paper \cite{CS} gave a renewed prominence to the Bessel operator  
in pde's and free boundaries. Among other things, they showed that, if for a given $a\in (-1,1)$ and a $u\in \mathscr S(\Rn)$, one indicates with $U(X)$, with $X = (x,z) \in \Rnp$, $x\in \Rn$, $z>0$, the solution to the Dirichlet problem
\begin{equation}\label{eep}
\begin{cases}
\mathscr L_a U = \operatorname{div}_X(z^a \nabla_X U) = 0\ \ \ \ \ \ \ \text{in}\ \Rnp,
\\
U(x,0) = u(x),
\end{cases}
\end{equation}
then with $s = \frac{1-a}2\in (0,1)$ the following Dirichlet-to-Neumann relation holds
\begin{equation}\label{dtn}
(-\Delta)^s u(x) = - \frac{2^{-a} \Gamma\left(\frac{1-a}2\right)}{\Gamma\left(\frac{1+a}2\right)} \underset{z\to 0^+}{\lim} z^a \frac{\p U}{\p z}(X).
\end{equation}
We note that when $s = 1/2$, then $a = 0$, and the operator in \eqref{eep} is the standard Laplacean in $\Rnp$.

The Caffarelli-Silvestre extension procedure \eqref{eep} has played a pivotal role in the analysis of nonlocal operators such as $(-\Delta)^s$, since via \eqref{dtn} it allows to turn  
problems involving the latter into ones involving the differential (local) operator  
$\mathscr L_a$. The Bessel operator occupies a central position in such procedure since the extension operator $\mathscr L_a$ can be written as follows
\begin{equation}\label{cs}
 \mathscr L_a = z^{a} (\Delta_x + \Ba).
\end{equation}
The reader should note the similarity between \eqref{ms}, \eqref{mol} and \eqref{cs}. In this perspective, one should think of the operator between parenthesis in right-hand side of \eqref{cs} as the standard Laplacean in the space $\R^{n+a+1}$ of fractal dimension $n+a+1$, with variables $(x,y)$, where $x\in \Rn$ and $y \in \R^{a+1}$, acting on a ``cylindrical" function $u(x,y) = f(x,z)$, with $z = |y|$.
In more recent years Stinga and Torrea  have generalized the extension procedure to different classes of operators, including uniformly elliptic operators in divergence form $L = \operatorname{div}(A(x) \nabla)$, see \cite{ST}, or the heat operator $H = \frac{\p}{\p t} - \Delta_x$, see \cite{ST2}. This latter result was also established simultaneously and independently by Nystr\"om and Sande in \cite{NS}.
    
Motivated by such developments, and also by the new ones in \cite{BG}, \cite{BDGP}, \cite{AT}, \cite{G18b}, \cite{BGMN} and \cite{GT}, in this paper we establish some properties of the Bessel semigroup, and provide some interesting applications of these results to the following degenerate parabolic operator in $\Rnp\times (0,\infty)$
\begin{equation}\label{heatext}
\p_t(z^a U) - \LL U = \p_t(z^a U) - \operatorname{div}_X(z^a \nabla_X U),
\end{equation}  
where $U = U(X,t)$ is a function defined in $\Rnp \times (0,\infty)$. Here, we have kept with the notations introduced before \eqref{eep} above. We mention that, similarly to its elliptic predecessor \eqref{eep}, the operator \eqref{heatext} is the extension operator for the fractional powers $(\p_t - \Delta)^s$, $0<s<1$, of the heat operator, see \cite{NS} and \cite{ST2}.

 A key remark concerning \eqref{heatext} is that it belongs to a general class of equations first introduced by Chiarenza and Serapioni in \cite{CS}. These authors considered degenerate parabolic equations in $\Rnn$ of the type 
\begin{equation}\label{csA}
 \p_t (\omega(x) u) - \operatorname{div}(A(x)\nabla u) = 0,
\end{equation}
where $\omega$ is a $A_2$-weight of Muckenhoupt in $\Rn$, and $A(x)$ is a matrix-valued function with bounded measurable coefficients, for which $A(x) = A(x)^T$, and such that 
\[
<A(x)\xi,\xi> \cong \omega(x) |\xi|^2.
\]
In their main result, they proved that nonnegative solutions of \eqref{csA} satisfy a scale invariant Harnack inequality on the standard parabolic cylinders. Such result proved to be the appropriate parabolic counterpart of the elliptic one previously obtained by Fabes, Kenig and Serapioni in \cite{FKS}. We note here that, 
since $\omega(X) = |z|^a$ belongs to $A_2(\Rnn)$ if and only if $|a|<1$, the model equation \eqref{heatext} is a special case of \eqref{csA}. 

The connection between \eqref{heatext} and the Bessel semigroup is in the fact that, similarly to \eqref{cs}, we can alternatively write the extension operator as
\[
\p_t(z^a U) - \LL U = z^a(\p_t - \Delta - \Ba).
\]
  
The parabolic operator \eqref{heatext} has recently received increasing attention. In connection with the parabolic Signorini problem, which is intimately linked to the obstacle problem for the fractional power $(\p_t - \Delta)^{1/2}$, the analysis of the case $a = 0$ in \eqref{heatext} was extensively developed in the monograph \cite{DGPT}. The general case $-1<a<1$ was studied in \cite{BG} in the problem of the unique continuation backward in time. In the paper \cite{ACM} the authors established the optimal interior regularity of the solutions of the thin obstacle problem for \eqref{heatext}. In \cite{AT} the authors obtained various interesting results on the nodal sets of the solutions of \eqref{heatext}. The paper \cite{G18b} studied the extension problem for hypoelliptic sub-Laplaceans of H\"ormander type. Finally, in the forthcoming article \cite{BDGP} the authors develop the analysis of the singular part of the free boundary in the thin obstacle problem studied in \cite{ACM}. Of course, this list of works is far from being exhaustive.

To introduce our results we recall that in their celebrated work  \cite{LY} Li and Yau proved (among other things) that if $f>0$ is a solution of the heat equation $\p_t f - \Delta f = 0$ on a boundaryless, complete $n$-dimensional Riemannan manifold $\mathbb M$ having $\operatorname{Ricci} \ge 0$, then its entropy $u = \log f$ satisfies the (deep) inequality on $\mathbb M\times (0,\infty)$,
\begin{equation}\label{LY0}
|\nabla u|^2 - \p_t u \le \frac{n}{2t}.
\end{equation}
We mention that  the inequality \eqref{LY0} becomes an equality when $f$ is the heat kernel in flat $\Rn$, see \eqref{heat} below.  The importance of the inequality \eqref{LY0} is underscored by the fact that a direct remarkable consequence of it is the following sharp form of the Harnack inequality, valid for any $x, y\in \mathbb M$ and any $0<s<t<\infty$,
\begin{equation}\label{harnack0}
f(x,s) \le f(y,t) \left(\frac ts\right)^{\frac n2} \exp\left(\frac{d(x,y)^2}{4t}\right).
\end{equation} 
Such Harnack inequality is keen to that proved independently by Hadamard \cite{Ha} and Pini \cite{Pi} for the standard heat equation in the plane. One remarkable aspect of \eqref{harnack0} is that the constant $\left(\frac ts\right)^{\frac n2} \exp\left(\frac{d(x,y)^2}{4t}\right)$ in its right-hand side is explicit and best possible.

In this note we start from a seemingly very simple problem. Namely, we consider the Cauchy problem for the Bessel operator $\Ba$ on the half-line $\{z>0\}$,
 with Neumann boundary condition (or  \emph{Feller's zero-flux condition}, see Section \ref{S:EM} below), 
\begin{equation}\label{CP0}
\begin{cases}
\de_t u  - \Ba u = 0,
\\
u(z,0) = \vf(z),
\\
\underset{z\to 0^+}{\lim} \paa u(z,t) = 0.
\end{cases}
\end{equation}
This corresponds to Brownian motion on the half-line $(0,\infty)$ reflected at $z = 0$, as opposed to killed Brownian motion, when a Dirichlet condition is imposed. We denote by $\{P^{(a)}_t\}_{t\ge 0}$ the semigroup associated with \eqref{CP0} and given by the formula \eqref{repfor} below. For the definition of the space $C_{(a)}^1(0,\infty)$ see Section \ref{S:besselsg} below. Our first main result is the following.

\begin{theorem}[Li-Yau type inequality]\label{T:LYgen}
Let $a\ge 0$. Given a function $\vf\ge 0$ such that $\vf\in \mathscr C_{(a)}^1(0,\infty)$, we have for any $z>0$ and $t>0$,
\begin{equation}\label{LYbessel0}
\left(\de_z \log P^{(a)}_t \vf(z)\right)^2 -  \de_t \log P^{(a)}_t \vf(z) <  \frac{a+1}{2t}.
\end{equation}
When $z = 0$ the inequality \eqref{LYbessel0}
is true for every $a>-1$ and with $\le$ instead of $<$.
\end{theorem}

Using Theorem \ref{T:LYgen} we then prove the following sharp result. 

\begin{theorem}[Harnack inequality]\label{T:harnackbessel0}
Let $a\ge 0$. For every $\vf\ge 0$ such that $\vf\in \mathscr C_{(a)}^1(0,\infty)$, we have for $z, \zeta \in \R^+$ and $0<s<t<\infty$,
\begin{equation}\label{harnackba}
P^{(a)}_s \vf(z) < P^{(a)}_t  \vf(\zeta) \left(\frac ts\right)^{\frac{a+1}2} \exp\left( \frac{(z-\zeta)^2}{4(t-s)}\right).
\end{equation}
\end{theorem}

The reader should notice the striking similarity between \eqref{LY0}, \eqref{harnack0} above, and \eqref{LYbessel0}, \eqref{harnackba}.
We emphasize that the factor $\frac{(z-\zeta)^2}{4(t-s)}$ in the exponential in the right-hand side of \eqref{harnackba} reflects the invariance of the heat operator $\de_t    - \Ba$ with respect to the dilations $(z,t) \to (\l z,\l^2 t)$, whereas the factor $\left(\frac ts\right)^{\frac{a+1}2}$ indicates that the number $Q = a+1$ plays the role of a fractal dimension for the semigroup $\{P^{(a)}_t\}_{t\ge 0}$. 

We also stress another aspect of Theorem \ref{T:harnackbessel0} that should not go unnoticed. Consider the quadrant $\mathscr Q^+ = \{(z,t)\in \R^2\mid z>0,\ t>0\}$. In every (elliptic or parabolic) Harnack inequality one expects the constant which multiplies the term in the right-hand side to blow-up as one approaches the boundary of the relevant domain. This does not happen for \eqref{harnackba}, as the factor $\left(\frac ts\right)^{\frac{a+1}2} \exp\left( \frac{(z-\zeta)^2}{4(t-s)}\right)$ does not seem to \emph{see} the vertical portion $\{(0,t)\in \R^2\mid t>0\}$ of $\p \mathscr Q^+$, exactly as for the global inequality \eqref{harnack0} above, in which there is no boundary. For instance, we can let $\zeta\to 0^+$, or even let $z, \zeta\to 0^+$, without causing the factor $\left(\frac ts\right)^{\frac{a+1}2} \exp\left( \frac{(z-\zeta)^2}{4(t-s)}\right)$ to blow-up. In other words, Theorem \ref{T:harnackbessel0} behaves like an \emph{interior} Harnack inequality in the whole half-plane $\mathscr Q  = \{(z,t)\in \R^2\mid z\in \R,\ t>0\}$. The explanation for this is in the vanishing Neumann condition in \eqref{CP0} above. Such condition implies that solutions to \eqref{CP0} be smooth in $z$ (in fact, real analytic) up to the vertical line $z = 0$. Therefore, if one defines $U(z,t) = u(|z|,t)$, one obtains a global solution on $\R\times (0,\infty)$ of the pde.

What is remarkable about Theorem \ref{T:LYgen} is that it ultimately hinges on the above inequality \eqref{soni0} for the Bessel quotient $y_{\nu}$. The key connection between Li-Yau and \eqref{soni0} is the identity \eqref{ly} in  Proposition \ref{P:LYp} below.
Since the link between $\nu$ and the parameter $a$ in \eqref{bes} is given by the equation $\nu = \frac{a-1}2$,
from our work in Section \ref{S:ly} it will be evident that, in its sharp form \eqref{LYbessel0} above, such inequality fails to hold when $-1<a< 0$, see also Section \ref{S:appendix}. 

Having said this, the question naturally arises of whether our approach can be pushed to establish a Harnack inequality for the semigroup $\{P^{(a)}_t\}_{t\ge 0}$ also in the range $-1<a<0$. We presently only have some inconclusive indication about the answer. Nonetheless, we emphasize that a Harnack inequality in the range $a\in (-1,0)$ is in fact known. For this one can invoke either Theorem 2.1 in \cite{CS}, or Theorem 4.1 in the more recent work \cite{EM3}. We have already discussed the Harnack inequality in \cite{CS}. In \cite{EM3} the authors prove a Harnack inequality for a general class of degenerate parabolic operators on manifolds with corners which arise in population biology. The Kimura equations treated in \cite{EM3} include as special case the model Cauchy problem \eqref{cpLb}, \eqref{zeroflux} below. As we show in Proposition \ref{P:equiv}, such model is equivalent to the problem \eqref{CP0}, and thus a Harnack inequality for the latter can be obtained from the cited \cite[Theorem 4.1]{EM3}.

We now discuss the second set of main results in this paper. Given a function $\vf\in C^\infty_0(\Rnp)$, consider the Cauchy problem with Neumann condition for the operator \eqref{heatext} above
\begin{equation}\label{cpext0}
\begin{cases}
\p_t(z^a U)  - \LL U = 0\ \ \ \ \ \ \ \ \ \ \ \  \text{in}\ \Rnp\times (0,\infty)
\\
U(X,0) = \vf(X),\ \ \ \ \ \ \ \ \ \ \ \ \ \ \ X\in \Rnp,
\\
\underset{z\to 0^+}{\lim} \paa U(x,z,t) = 0.
\end{cases}
\end{equation}
The solution to \eqref{cpext0} is represented by the formula
\begin{equation}\label{solform0}
U(X,t) = \Pa \vf(X) =  \int_{\Rnp} \vf(Y) \mathscr G_a(X,Y,t) \zeta^a dY,
\end{equation}
where $\mathscr G_a(X,Y,t)$ is given in \eqref{sGa} below. We note that $\{\Pa\}_{t\ge 0}$ defines a semigroup on $C^\infty_0(\Rnp)$.
Concerning such semigroup we have the following sharp Harnack inequality.

\begin{theorem}\label{T:harext0}
Let $a\ge 0$. Let $\vf \ge 0$ be a function for which $U$ given by \eqref{solform0} represents a classical solution to \eqref{cpext0}. Then, for every $X, Y\in \Rnp$ and every $0<s<t<\infty$, we have
\begin{equation}\label{harext0}
U(X,s) <  U(Y,t) \left(\frac ts\right)^{\frac{n+a+1}2} \exp\left(\frac{|X-Y|^2}{4(t-s)}\right).
\end{equation}
When $X = (x,0), Y = (y,0)$ the inequality is valid in the full range $a>-1$, and becomes
\begin{equation}\label{harext02}
U(x,0,s) \le U(y,0,t) \left(\frac ts\right)^{\frac{n+a+1}2} \exp\left(\frac{|x-y|^2}{4(t-s)}\right).
\end{equation}
\end{theorem}
Once again, one should compare Theorem \ref{T:harext0} to the Harnack inequality \eqref{harnack0} of Li and Yau. Concerning Theorem \ref{T:harext0} two comments are in order:
\begin{itemize}
\item[1)] For any $a\in (-1,1)$ a Harnack inequality for positive solutions of \eqref{cpext0} on the parabolic cylinders $B(r) \times (\alpha r^2,\beta r^2)$, where $B(r)$ is a Euclidean ball in $\Rnn$, can be obtained from Theorem 2.1 in \cite{CS}. One must first prove that the Neumann condition \eqref{cpext0} above implies that $U$ is smooth in $z$ up to the thin manifold  $\mathbb M = (\Rn\times \{0\})\times (0,\infty)$, and then show that the even reflection of $U$ in $z$ is a solution across such manifold. At that point, one can appeal to the above cited interior Harnack inequality in \cite{CS}. The novelty in Theorem \ref{T:harext0} with respect to such approach is that it produces the explicit sharp constant $\left(\frac ts\right)^{\frac{n+a+1}2} \exp\left(\frac{|X-Y|^2}{4(t-s)}\right)$. Furthermore, our direct proof has already encoded the information of being an ``interior" Harnack inequality, and we do not need to even reflect the solution across the thin manifold $\mathbb M$. In fact, with $X = (x,z)$ and $Y = (y,\zeta)$, we can let either $z$ or $\zeta$, or both tend to zero, and yet the constant $\exp\left(\frac{|X-Y|^2}{4(t-s)}\right) = \exp\left(\frac{|x-y|^2 + (z-\zeta)^2}{4(t-s)}\right)$ does not blow up, see \eqref{harext02} above.  
\item[2)] Theorem \ref{T:harext0} is valid in the whole range $a\ge 0$, whereas Theorem 2.1 in \cite{CS} is not applicable when $a\ge 1$, since in such range $\omega(X) = |z|^a$ is not even locally integrable.
\end{itemize} 

The proof of Theorem \ref{T:harext0} is based on the following inequality of Li-Yau type.

\begin{theorem}\label{T:LYext0}
Let $a\ge 0$ and $\vf$ and $U$ be as in Theorem \ref{T:harext0}. Then, for any $X\in \Rnp$ and $t>0$ one has
\begin{equation}\label{LYext}
|\nabla_X \log U(X,t)|^2 - \p_t \log U(X,t) <  \frac{n+a+1}{2t}.
\end{equation}
If instead $X = (x,0), Y = (0,y)\in \Rn\times\{0\}$, then the inequality \eqref{LYext} is true for every $a>-1$, in the following form
\begin{equation}\label{LYext2}
|\nabla_X \log U(x,0,t)|^2 - \p_t \log U(x,0,t) \le  \frac{n+a+1}{2t}.
\end{equation}
\end{theorem} 
It is remarkable that, similarly to that of Theorem \ref{T:LYgen}, also  Theorem \ref{T:LYext0} ultimately rests on the elementary inequality \eqref{soni0} above.

The final set of results in this paper has to do with monotonicity formulas. In the paper \cite{BG} the authors have studied the problem of strong unique continuation backward in time for the nonlocal equation
\[
(\p_t - \Delta)^s u = V(x,t) u,\ \ \ \ \ \ \ \ \ 0<s<1.
\]
Given $s$ in such range, let $a = 1-2s$. One of the central results in \cite{BG} was a monotonicity formula for solutions of the extension operator \eqref{heatext} above, with the Neumann condition 
\[
- \frac{2^{-a} \Gamma\left(\frac{1-a}2\right)}{\Gamma\left(\frac{1+a}2\right)}\underset{z\to 0^+}{\lim} \paa U(x,z,t) = V(x,t) u(x,t).
\]
Monotonicity formulas for the heat equation go back to the work of Struwe \cite{struwe}. In \cite{Po} Poon first proved a monotonicity formula for the heat equation analogous to the celebrated one established by Almgren in \cite{A} for multiple-valued harmonic functions. The work \cite{DGPT} contains, among other things, generalizations of Poon's result to the parabolic Signorini problem. This is the case $a = 0$ of the the extension operator \eqref{heatext} above. The forthcoming article \cite{BDGP} extends the results in \cite{DGPT} to the full range $a\in (-1,1)$. 

We emphasize that the problems studied in \cite{DGPT} and \cite{BDGP} are thin obstacle problems for the operator \eqref{heatext} in which the free boundary lives in the thin manifold $\{z=0\}$.
Because of this reason, in all the results in these papers the various quantities at play are ``centered" at one convenient point, the origin, of the thin manifold. By this we mean that the relevant Gaussian measures in the relevant monotonicity formulas are centered at the origin. Such choice is for all practical purposes immaterial if one deals with a problem in which the focus is the thin manifold. But it becomes relevant in situations when this is no longer the case. 

Such considerations lead to the question of the stability of monotonicity properties in dependence of the center of the Gaussian measure. A natural testing ground in this direction is the prototypical non-translation invariant parabolic pde in \eqref{CP0} above. We have discovered that, interestingly, there is a discrepancy in the resulting monotonicity formulas, according to whether the ``center" of the relevant Gaussian measure is located at a point $(\zeta,t)$ with $\zeta >0$, or $\zeta = 0$. Such discrepancy is similar to that in Theorems \ref{T:LYgen} and \ref{T:harext0}, but it no longer rests on \eqref{soni0} above. Remarkably, the monotonicity of the relevant energy and frequency functions now ultimately depends on the  monotonicity of the Bessel quotient $y_\nu$. For the proof of such property see Proposition \ref{P:monotonicityy} below.

For the definition of the scaled energy,  with respect to the backward Gaussian-Bessel measure centered at $(z,T)$, $E^{(a)}_{z,T}(t)$, the reader should see \eqref{e1} below. The following are our main results.

\begin{theorem}[Struwe type monotonicity formula]\label{T:struwe0}
Suppose that $u$ be a solution to \eqref{eu} in $(0,\infty)\times (0,\infty)$ satisfying the condition \eqref{snc} (and other reasonable growth assumptions). Then, for any fixed $z>0$ the function $t\to E^{(a)}_{z,T}(t)$ is strictly decreasing on $(0,T)$ when $a\ge 0$. Precisely, we have
\begin{equation}\label{E's}
\frac{d E^{(a)}_{z,T}}{dt}(t) = - (T-t) \int_0^\infty \bigg[u_{t} + u_\zeta \frac{p^{(a)}_\zeta}{p^{(a)}}\bigg]^2 p^{(a)} \zeta^a d\zeta + G^{(a)}_{z,T}(t),
\end{equation}
where
\begin{equation}\label{G0}
G^{(a)}_{z,T}(t) = - (T-t) \int_0^\infty u_\zeta^2 \left[\frac{p^{(a)}_{\zeta\zeta}}{p^{(a)}} - \frac{(p^{(a)}_\zeta)^2}{(p^{(a)})^2} + \frac{1}{2(T-t)}\right] p^{(a)} \zeta^a d\zeta < 0.
\end{equation}
When $z = 0$, then for any $a> - 1$ we have $G^{(a)}_T(t) = G^{(a)}_{0,T}(t) \equiv 0$, and the function $t\to E^{(a)}_T(t)$ is monotone decreasing on $(0,T)$.
\end{theorem}

\begin{remark}\label{R:discrepancy}
We emphasize the remarkable discrepancy in Theorem \ref{T:struwe0} between the case $z>0$, in which the (strict) monotonicity of $t\to E_{z,T}^{(a)}(t)$ holds only when $a\ge 0$, and that when $z = 0$, in which we have monotonicity in the full range $a>-1$.
\end{remark}

The next theorem is the second main result about monotonicity formulas. For the meaning of the frequency centered at $z\ge 0$, $N_z^{(a)}(r)$, we refer the reader to Definition \ref{D:frequency} below. When $z = 0$ we simply write $N^{(a)}(r)$.

\begin{theorem}[Poon type monotonicity formula]\label{T:poon0}
Let $u$ be a solution to \eqref{eu}, satisfying \eqref{snc}. For any $z>0$ the frequency $r\to N_z^{(a)}(r)$ is strictly increasing when $a\ge 0$. If instead $z = 0$, then the frequency is non-decreasing for any $a>-1$. Furthermore, in this second case we have $N^{(a)}(r) \equiv \kappa$ if and only if $u$ is homogeneous of degree $\kappa$ with respect to the parabolic dilations $(\zeta,t)\to (\la \zeta,\la^2 t)$.
\end{theorem}

\vskip 0.2in

\noindent \textbf{Acknowledgment:} I thank Camelia Pop for her interest in the present work, for her kindness (and patience!) in educating me in the probabilistic aspects of the Bessel semigroup, and for bringing to my attention the references \cite{EM}, \cite{EM2}, \cite{EM3}  and \cite{EP}. In Section \ref{S:EM} below I discuss the connection between Theorem \ref{T:harnackbessel0} above and the results in \cite{EM3}, \cite{EP} and \cite{CS}. I also mention that Proposition \ref{P:zero} below came up in a conversation with Giulio Tralli and I thank him for an interesting exchange. I am also grateful to Javier Segura and Bettina Gr\"un who have kindly corresponded with me. In an early stage of this article, Javier first provided numerical evidence that Proposition \ref{P:soni} fails in the range $-1<\nu<-1/2$. Finally, I thank Charles Epstein and Rafe Mazzeo for their kind feedback.


\section{The Bessel semigroup}\label{S:besselsg}

In this section we collect some known facts concerning the Cauchy problem for the Bessel operator 
 with Neumann boundary condition  
\begin{equation}\label{CP}
\begin{cases}
\p_t u  - \Ba u = 0,
\\
u(z,0) = \vf(z),
\\
\underset{z\to 0^+}{\lim} \paa u(z,t) = 0.
\end{cases}
\end{equation}

We begin by introducing the following classes of functions
\[
\mathscr C_{(a)}(0,\infty) = \left\{\vf\in C(0,\infty) \mid \int_0^R |\vf(z)| z^{a} dz < \infty,
\ \ \int_R^\infty |\vf(z)| z^{\frac a2} dz < \infty, \forall R>0\right\},
\]
and 
\[
\mathscr C^1_{(a)}(0,\infty) = \left\{\vf\in C^1(0,\infty)\mid \vf, \frac{1}{z} \vf' \in \mathscr C_{(a)}(0,\infty)\right\}.
\]
As it was observed in (22.8) of \cite{G}, membership in $\mathscr C^1_{(a)}(0,\infty)$ imposes, in particular, the weak \emph{Neumann condition} 
\begin{equation}\label{weakn}
\underset{z\to 0^+}{\liminf}\ z^{a} |\vf'(z)| = 0.
\end{equation}

\begin{prop}\label{P:CP}
Given $\vf\in\mathscr C^1_{(a)}(0,\infty)$, the Cauchy problem \eqref{CP} admits the following solution 
\begin{equation}\label{repfor}
u(z,t) = P^{(a)}_t \vf(z) \overset{def}{=} \int_0^\infty \vf(\zeta) p^{(a)}(z,\zeta,t) \zeta^a d\zeta,
\end{equation}
where for $z,\zeta,t>0$ we have denoted by
\begin{align}\label{fs}
p^{(a)}(z,\zeta,t) & =(2t)^{-\frac{a+1}{2}}\left(\frac{z\zeta}{2t}\right)^{\frac{1-a}{2}}I_{\frac{a-1}{2}}\left(\frac{z\zeta}{2t}\right)e^{-\frac{z^2+\zeta^2}{4t}}
\\
& = \frac{1}{2t} (z\zeta)^{\frac{1-a}{2}}I_{\frac{a-1}{2}}\left(\frac{z\zeta}{2t}\right)e^{-\frac{z^2+\zeta^2}{4t}},
\notag
\end{align}
the heat kernel of $\Ba$ on $(\R^+,z^a dz)$, with Neumann boundary conditions. For $t\le 0$ we set $p^{(a)}(z,\zeta,t) \equiv 0$.
\end{prop}

In \eqref{fs} we have denoted by $I_\nu(z)$ the modified Bessel function of the first kind and order $\nu\in \mathbb C$ defined  by the series \eqref{I} below. Formulas \eqref{repfor}, \eqref{fs} are well-known to workers in probability (see for instance formula (6.14) on p. 238 in \cite{KaT}), but not equally known to those in partial differential equations. For a direct proof based exclusively on analytic tools we refer the reader to Proposition 22.3 in \cite{G}. Another analytical proof can be found in the paper \cite{EM}, where Epstein and Mazzeo construct the fundamental solution \eqref{kb} for the Cauchy problem \eqref{cpLb} below. For this aspect we refer the reader to Section \ref{S:EM}. 

We next collect some important properties of the Bessel heat kernel $p^{(a)}(z,\zeta,t)$ in \eqref{fs} above. Since we have not found in the literature a direct source which is suitable for workers in analysis, we provide details of their proofs. We begin by noting the following simple facts:
\begin{itemize}
\item[(i)] $p^{(a)}(z,\zeta,t)>0$ for every $z, \zeta>0$ and $t>0$;
\item[(ii)] $p^{(a)}(z,\zeta,t) = p^{(a)}(\zeta,z,t)$;
\item[(iii)] $p^{(a)}(\la z,\la \zeta,\la^2 t) = \la^{-(a+1)} p^{(a)}(\zeta,z,t)$.
\end{itemize}
Property (i) follows from the fact that $I_\nu(z)>0$ for any $z>0$, and any $\nu \ge - 1$, see \eqref{I} and the comments following it. Property (ii) is obvious from \eqref{fs} and it is a reflection of the symmetry of the operator $\Ba$ on $(0,\infty)$ equipped with the measure $d\mu^{(a)}(z) = z^a dz$. Property (iii) reflects the invariance of the heat operator $\p_t - \Ba$ with respect to the parabolic scalings $\la \to (\la z,\la^2 t)$. In particular, (iii) implies that
\[
p^{(a)}(z,\zeta,t) = t^{-\frac{a+1}2} p^{(a)}(\frac{z}{\sqrt t},\frac{\zeta}{\sqrt t},1).
\]

From \eqref{fs} and \eqref{smallz} below we see that for any $\zeta>0$ and $t>0$,
\begin{equation}\label{pazero}
p^{(a)}(0,\zeta,t) = \underset{z\to 0^+}{\lim} p^{(a)}(z,\zeta,t) = \frac{1}{2^{a} \G(\frac{a+1}2)} t^{-\frac{a+1}{2}} e^{-\frac{\zeta^2}{4t}}.
\end{equation}
We also have for any $\zeta>0$ and $t>0$,
\begin{equation}\label{dzpazero}
\underset{z\to 0^+}{\lim} z^a \p_z p^{(a)}(z,\zeta,t) = 0.
\end{equation}
The limit relation \eqref{dzpazero} can be proved using \eqref{derz} and \eqref{zerozero} below. Since $p^{(a)}(z,\zeta,t) $ is the Neumann fundamental solution for problem \eqref{CP}, the property \eqref{dzpazero} should come as no surprise.

\begin{remark}
We note explicitly that although $\Ba$, originally defined on $C^\infty_0(0,\infty)$, is symmetric with respect to the measure $d\mu_{(a)} = z^a dz$ for every $a>-1$, it is essentially self-adjoint only when either $a<0$, or $a>2$. For this see \cite[Proposition 2.4.1]{BGL}. 
\end{remark}

If we fix $z>0, t>0$, then the asymptotic behavior of $p^{(a)}(z,\zeta,t)\zeta^a$ as $\zeta \to 0^+$, or $\zeta \to \infty$, follows from that of the Bessel function $I_\nu$. Keeping in mind that \eqref{smallz} and \eqref{ab2} give
\begin{equation}\label{ab0infty}
I_{\frac{a-1}2}(z) \cong \begin{cases}
\frac{z^{\frac{a-1}2}}{2^{\frac{a-1}2} \G(\frac{a+1}2)},\ \ \ \ \ \ \ \ \ \ \ \ \ \ \ \ \ \ \text{as}\ z\to 0^+,
\\
\frac{e^z}{(2\pi z)^{1/2}} \left(1+ O(|z|^{-1})\right),\ \ \ \text{as}\ z\to \infty,
\end{cases}
\end{equation}
from \eqref{fs} and \eqref{ab0infty} we see that for every fixed $(z,t)\in (0,\infty)\times (0,\infty)$
\begin{equation}\label{ab0inftypa}
p^{(a)}(z,\zeta,t)\zeta^a = \begin{cases}
O(\zeta^a),\ \ \ \ \ \ \ \ \ \ \ \ \ \ \ \text{as}\ \zeta\to 0^+,
\\
O(\zeta^{\frac{a}2} e^{-\frac{(\zeta - z)^2}{4t}}),\ \ \ \text{as}\ \zeta\to \infty,
\end{cases}
\end{equation}
Since $a>-1$ we infer from \eqref{ab0inftypa} that for every fixed $(z,t)\in (0,\infty)\times (0,\infty)$,
\[
\int_0^\infty p^{(a)}(z,\zeta,t)\zeta^a d\zeta < \infty.
\]
Thus, it is possible to consider $P^{(a)}_t \vf$ for every $\vf\in L^\infty(0,\infty)$. In particular, it makes sense to consider $P^{(a)}_t 1$. The next result provides an important information in this connection. 

\begin{prop}[Stochastic completeness]\label{P:sc}
Let $a>-1$. For every $(z,t)\in (0,\infty)\times (0,\infty)$ one has
\[
P^{(a)}_t 1(z) = \int_0^\infty p^{(a)}(z,\zeta,t) \zeta^a d\zeta = 1.
\]
\end{prop}

\begin{proof}
We have
\begin{align*}
& \int_0^\infty p^{(a)}(z,\zeta,t) \zeta^a d\zeta = (2t)^{-\frac{a+1}{2}} e^{-\frac{z^2}{4t}} \int_0^\infty \left(\frac{z\zeta}{2t}\right)^{\frac{1-a}{2}}I_{\frac{a-1}{2}}\left(\frac{z\zeta}{2t}\right)e^{-\frac{\zeta^2}{4t}} \zeta^{a+1} \frac{d\zeta}{\zeta}
\\
& \left(\text{change of variable}\ y =  \frac{z\zeta}{2t},\ \text{so that}\ \zeta = \frac{2t}{z} y\right)
\\
& = \frac{(2t)^{-\frac{a+1}{2}} e^{-\frac{z^2}{4t}} 2^{a+1} t^{a+1}}{z^{a+1}} \int_0^\infty y^{\frac{a+1}{2}}I_{\frac{a-1}{2}}(y)e^{-\frac{t}{z^2}y^2} dy
\\
& = \frac{e^{-\frac{z^2}{4t}} 2^{\frac{a+1}2} t^{\frac{a+1}2}}{z^{a+1}} \int_0^\infty y^{\frac{a+1}{2}}I_{\frac{a-1}{2}}(y)e^{-\frac{t}{z^2}y^2} dy
\end{align*}
If we set
\[
\nu = \frac{a-1}2,
\]
then $\nu>-1$ and $\nu + 1 = \frac{a+1}2$. Applying Lemma \ref{L:weber} below with such choice of $\nu$ and  $\alpha =  \frac{t}{z^2}$, we find
\[
 \int_0^\infty y^{\frac{a+1}{2}}I_{\frac{a-1}{2}}(y)e^{-\frac{t}{z^2}y^2} dy = \frac{e^{1/4\alpha}}{2^{\nu+1} \alpha^{\nu+1}} = \frac{e^{\frac{z^2}{4t}} z^{a+1}}{2^{\frac{a+1}2} t^{\frac{a+1}2}}.
 \]
This proves the proposition.

\end{proof}

We next prove that \eqref{repfor} defines a semigroup of operators.

\begin{prop}[Chapman-Kolmogorov equation]\label{P:semigroup}
Let $a>-1$. For every $z, \eta>0$ and every $0<s, t<\infty$ one has
\[
p^{(a)}(z,\eta,t) = \int_0^\infty p^{(a)}(z,\zeta,t) p^{(a)}(\zeta,\eta,s) \zeta^a d\zeta.
\]
\end{prop}

\begin{proof}
We begin with the right-hand side in the above equation
\begin{align*}
& \int_0^\infty p^{(a)}(z,\zeta,t) p^{(a)}(\zeta,\eta,s) \zeta^a d\zeta 
\\
& = (2t)^{-1} (2s)^{-1} (z\eta)^{\frac{1-a}{2}} e^{-\frac{z^2}{4t}} e^{-\frac{\eta^2}{4s}}  \int_0^\infty \zeta^{1-a}I_{\frac{a-1}{2}}\left(\frac{z\zeta}{2t}\right) I_{\frac{a-1}{2}}\left(\frac{\eta\zeta}{2s}\right)e^{-\frac{\zeta^2}{4t}} e^{-\frac{\zeta^2}{4s}}\zeta^a d\zeta
\\
& = (2t)^{-1} (2s)^{-1} (z\eta)^{\frac{1-a}{2}} e^{-\frac{z^2}{4t}} e^{-\frac{\eta^2}{4s}}  \int_0^\infty \zeta e^{-\left(\frac{1}{4t} + \frac{1}{4s}\right) \zeta^2}  I_{\frac{a-1}{2}}\left(\frac{z\zeta}{2t}\right) I_{\frac{a-1}{2}}\left(\frac{\eta\zeta}{2s}\right)  d\zeta
\end{align*}
At this point we invoke the following formula, see e.g. no. 8 on p. 321 in \cite{PBM},
\begin{equation}\label{pbm}
\int_0^\infty \zeta e^{-p \zeta^2} I_\nu(b \zeta) I_\nu(c\zeta) d\zeta = \frac{1}{2p} e^{\frac{b^2+c^2}{4p}} I_\nu\left(\frac{bc}{2p}\right),\ \ \ \ \ \ \ \ \nu> - 1,\ \Re p>0.
\end{equation}
Applying \eqref{pbm} with $\nu = \frac{a-1}{2}$ and 
\[
p = \frac{1}{4t} + \frac{1}{4s} = \frac{t+s}{4ts},\ \ \ \ \  \ b = \frac{z}{2t},\ \ \ \ c = \frac{\eta}{2s},
\]
we find after some elementary reductions
\begin{align*}
&  \int_0^\infty \zeta e^{-\left(\frac{1}{4t} + \frac{1}{4s}\right) \zeta^2}  I_{\frac{a-1}{2}}\left(\frac{z\zeta}{2t}\right) I_{\frac{a-1}{2}}\left(\frac{\eta\zeta}{2s}\right)  d\zeta  = \frac{2ts}{t+s} e^{\frac{s}{4t(t+s)} z^2} e^{\frac{t}{4s(t+s)} \eta^2} I_{\frac{a-1}{2}}\left(\frac{z \eta}{2(t+s)}\right).
\end{align*}
Substituting this identity in the above equation and simplifying, we obtain
\begin{align*}
& \int_0^\infty p^{(a)}(z,\zeta,t) p^{(a)}(\zeta,\eta,s) \zeta^a d\zeta 
\\
& = (2t)^{-1} (2s)^{-1} (z\eta)^{\frac{1-a}{2}} e^{-\frac{z^2}{4t}} e^{-\frac{\eta^2}{4s}}  \frac{2ts}{t+s} e^{\frac{s}{4t(t+s)} z^2} e^{\frac{t}{4s(t+s)} \eta^2} I_{\frac{a-1}{2}}\left(\frac{\zeta \eta}{2(t+s)}\right)
\\
& = \frac{1}{2(t+s)} (z\eta)^{\frac{1-a}{2}} e^{-\frac{z^2}{4t}\left(1-\frac{s}{t+s}\right)} e^{-\frac{\eta^2}{4s}\left(1-\frac{t}{t+s}\right)} I_{\frac{a-1}{2}}\left(\frac{\zeta \eta}{2(t+s)}\right)
\\
& = \frac{1}{2(t+s)} (z\eta)^{\frac{1-a}{2}} e^{-\frac{z^2 + \eta^2}{4(t+s)}} I_{\frac{a-1}{2}}\left(\frac{\zeta \eta}{2(t+s)}\right) = p^{(a)}(z,\eta,t).
\end{align*}

An immediate consequence of Proposition \ref{P:semigroup} is the following.

\begin{prop}[Semigroup property]\label{C:semigroup}
For every $a>-1$ and every $t, s >0$ one has
\[
P^{(a)}_t \circ P^{(a)}_s = P^{(a)}_{t+s}.
\]
\end{prop}

\end{proof}

We close this section by analyzing explicitly the case in which $a = 0$ in \eqref{fs}.
In such case we have
\begin{equation}\label{fs0}
p^{(0)}(z,\zeta,t)=(2t)^{-\frac{1}{2}}\left(\frac{z\zeta}{2t}\right)^{\frac{1}{2}}I_{- 1/2}\left(\frac{z\zeta}{2t}\right)e^{-\frac{z^2+\zeta^2}{4t}},
\end{equation}
We now note the following well-known formulas,  see (5.8.5) on p. 112 in \cite{Le}, 
\begin{equation}\label{half}
I_{-1/2}(z) = \left(\frac{2}{\pi z}\right)^{1/2} \cosh z,\ \ \ \ \ \ \ \ \ I_{1/2}(z) = \left(\frac{2}{\pi z}\right)^{1/2} \sinh z.
\end{equation}
Using \eqref{half} in \eqref{fs0}, we obtain
\begin{align}\label{fs02}
p^{(0)}(z,\zeta,t) & =(2t)^{- 1/2}\left(\frac{z\zeta}{2t}\right)^{\frac{1}{2}}\frac{\sqrt{4t}}{\sqrt \pi \sqrt{z\zeta}} \cosh\left(\frac{z\zeta}{2t}\right)e^{-\frac{z^2+\zeta^2}{4t}}
\\
& = (\pi t)^{- 1/2} \cosh\left(\frac{z\zeta}{2t}\right)e^{-\frac{z^2+\zeta^2}{4t}}
\notag\\
& = (4\pi t)^{- 1/2} \left(e^{\frac{z\zeta}{2t}} + e^{-\frac{z\zeta}{2t}}\right)  e^{-\frac{z^2+\zeta^2}{4t}}
\notag\\
& = (4\pi t)^{- 1/2} \left(e^{-\frac{(z-\zeta)^2}{4t}} + e^{-\frac{(z+\zeta)^2}{4t}}\right),
\notag
\end{align}
where in the last equality we have used the identities
\begin{equation}\label{ide}
e^{-\frac{(z-\zeta)^2}{4t}} = e^{-\frac{z^2+\zeta^2}{4t}} e^{\frac{z\zeta}{2t}},\ \ \ \ e^{-\frac{(z+\zeta)^2}{4t}} = e^{-\frac{z^2+\zeta^2}{4t}} e^{-\frac{z\zeta}{2t}}.
\end{equation}
From \eqref{repfor} we have
\begin{align}\label{repfor2}
& P^{(0)}_t \vf(z) = (4\pi t)^{- 1/2} \int_0^\infty \vf(\zeta) e^{-\frac{(z-\zeta)^2}{4t}} d\zeta + (4\pi t)^{- 1/2} \int_0^\infty \vf(\zeta) e^{-\frac{(z+\zeta)^2}{4t}} d\zeta
\\
& = (4\pi t)^{- 1/2} \int_0^\infty \vf(\zeta) e^{-\frac{(z-\zeta)^2}{4t}} d\zeta + (4\pi t)^{- 1/2} \int_{-\infty}^0 \vf(-\zeta) e^{-\frac{(z-\zeta)^2}{4t}} d\zeta
\notag\\
& = (4\pi t)^{- 1/2} \int_{-\infty}^\infty \Phi(\zeta) e^{-\frac{(z-\zeta)^2}{4t}} d\zeta,
\notag
\end{align}
where we have let 
\begin{equation}\label{ee}
\Phi(\zeta) = \vf(|\zeta|),\ \ \ \ \ \ \ \ \zeta\in \R.
\end{equation}
Formula \eqref{repfor2} proves the following result.

\begin{prop}\label{P:zero}
When $a=0$, at any $z, t >0$ the solution $u(z,t)$ to the problem \eqref{CP} coincides with the function $U(z,t)$ that solves the problem
\begin{equation}\label{CPR}
\begin{cases}
\de_t U  - \de_{zz} U  = 0\ \ \ \ \ \ \ \ \ \ \text{in}\ \R,
\\
U(z,0) = \Phi(z),\ \ \ \ \ \ \ \ \ \ z\in \R,
\end{cases}
\end{equation}
where $\Phi$ is the even extension \eqref{ee} to the whole line $\R$ of the function $\vf$ on $[0,\infty)$.
\end{prop}

In connection with the Dirichlet Bessel semigroup, we mention the papers \cite{BM1}, \cite{BM2} in which the authors establish various sharp asymptotic bounds for the heat kernel.


\section{A curvature-dimension inequality}\label{S:bessel}

In the famous paper \cite{BE} Bakry and \'Emery introduced their so-called $\G$-calculus as a different way of approaching global results in geometry through analytical tools. At the roots of such calculus there is the notion of \emph{curvature-dimension} inequality. A $n$-dimensional Riemannian manifold $\mathbb M$ with Laplacean $\Delta$ is said to satisfy the curvature-dimension inequality CD$(\rho,n)$ for some $\rho\in \R$ if for all functions $f\in C^\infty(\mathbb M)$ one has
\begin{equation}\label{cd}
\G_2(f) \ge \frac{1}{n} (\Delta f)^2 + \rho \G(f).
\end{equation}
Here, $\G$ and $\G_2$ respectively denote the \emph{carr\'e du champ} and the \emph{Hessian} canonically associated with $\Delta$, see \cite{BE}, and also the book \cite{BGL}. A remarkable aspect of \eqref{cd} is that it is equivalent to the lower bound Ricci $\ge \rho$ on the Ricci tensor of $\mathbb M$. In the paper \cite{BGjems} it was shown that many global properties of the heat semigroup, in a setting which includes the Riemannian one, can be derived exclusively from a generalization of the curvature-dimension inequality \eqref{cd}. In this connection one should also see \cite{BGimrn}, \cite{BBG} and \cite{BBGM}.

In this section we observe that the Bessel semigroup on $(\R_+,d\mu^{(a)})$, where $d\mu^{(a)}(z) = z^{a} dz$, with generator $\Ba$, satisfies a property similar to \eqref{cd} provided that $a\ge 0$, see Proposition \ref{P:bes} below. Although we do not use such fact in the present paper, we have decided to include it since, interestingly, it displays the same ``best possible" nature of the Bessel process  which permeates all our results.

We begin with defining for every $f, g\in C^\infty(\R)$ the \emph{carr\'e du champ} associated with $\Ba$,
\begin{equation}\label{besG}
\ga(f,g) = \frac 12\left[\Ba(fg) - f \Ba g - g \Ba f\right].
\end{equation}
One easily verifies that
\begin{equation}\label{besG2}
\ga(f,g) = f' g',\ \ \ \ \ \ \ \ \text{hence}\ \ \ \ \ \ \ \ga(f) = (f')^2.
\end{equation}
Next, we consider the Hessian associated with $\Ba$
\begin{equation}\label{gga}
\gga(f,g) = \frac 12 \left[\Ba \ga(f,g) - \ga(f,\Ba g) - \ga(g,\Ba f)\right].
\end{equation}
A simple calculation shows that
\begin{equation}\label{gga2}
\gga(f,g) = f'' g'' + \frac{a}{z^2} f' g', \ \ \ \ \ \ \text{therefore}\ \ \ \ \ \ \ \ \gga(f) = (f'')^2 + \frac{a}{z^2} (f')^2. 
\end{equation}

We can now establish the relevant curvature-dimension inequality for the Bessel semigroup. The reader should also see \cite{BGL}, especially (1.16.9) and the discussion in Section 2.4.2.

\begin{prop}\label{P:bes}
For every function $f\in C^\infty(\R_+)$ one has
\begin{equation}\label{cda}
\gga(f) \ge \frac{1}{a+1} (\Ba f)^2,
\end{equation}
if and only if $a\ge 0$. In other words, $\Ba$ satisfies the curvature-dimension inequality \emph{CD}$(0,a+1)$ on $\R_+$ if and only if $a\ge 0$. 
\end{prop}

\begin{proof}
If we start from assuming that $a>-1$, then $a+1>0$, and therefore in view of \eqref{gga2} the desired conclusion is equivalent  to
\[
\left(f'' + \frac{a}z f'\right)^2  \le  (a+1)\left((f'')^2 + \frac{a}{z^2} (f')^2\right).
\]
In turn, this inequality is equivalent to
\[
a \left(f'' - \frac 1z f'\right)^2 \ge 0,
\]
which is true for any $f\in C^\infty(\R)$ if and only if $a\ge 0$. When $-1<a <0$ the inequality \eqref{cda} gets reversed.

\end{proof}

We note in closing that, when $a>0$, then the fractal ``dimension" $Q = a+1$ in the inequality CD$(0,a+1)$ in \eqref{cda} above is strictly bigger than the topological dimension of the ambient manifold $M = \R^+$.


\section{An inequality of Li-Yau type for the Bessel semigroup}\label{S:ly}

In this section we prove Theorem \ref{T:LYgen} above. As we have mentioned, such result ultimately hinges on the global property \eqref{soni0} of the modified Bessel function $I_\nu$.

The following proposition represents the Bessel semigroup counterpart of the simple (but important) fact that for the standard heat kernel $p(x,y,t) = (4\pi t)^{-n/2} \exp(- \frac{|x-y|^2}{4t})$ in $\Rn$, we have
\begin{equation}\label{heat}
|\nabla_x \log p(x,y,t)|^2 - \p_t \log p(x,y,t) = \frac{n}{2t}.
\end{equation}
Except that, as \eqref{ly22} and \eqref{ly2} in Propositions \ref{P:LYp} and  \ref{C:ly}
show, one should not expect the equality as in \eqref{heat}. 

Hereafter, for $\nu>-1$ we indicate with $y_\nu(z) = I_{\nu+1}(z)/I_\nu(z)$ the Bessel quotient, see \eqref{bq} below. For a detailed analysis of the function $y_\nu$ we refer the reader to Section \ref{S:appendix}.

\begin{prop}\label{P:LYp}
Let $a > -1$. For every $z, \zeta\in \R^+$ and $t>0$ one has
\begin{align}\label{ly}
& \left(\de_z \log p^{(a)}(z,\zeta,t)\right)^2 - \de_t \log p^{(a)}(z,\zeta,t) = \frac{a+1}{2t} + \frac{\zeta^2}{4t^2} \bigg(y_{\frac{a-1}{2}}\big(\frac{z\zeta}{2t}\big)^2 - 1\bigg).
\end{align}
In particular, if we let $z\to 0^+$ in \eqref{ly} we obtain for any $\zeta>0$ and $t>0$,
\begin{equation}\label{ly22}
\left(\de_z \log p^{(a)}(0,\zeta,t)\right)^2 - \de_t \log p^{(a)}(0,\zeta,t)  < \frac{a+1}{2t}.
\end{equation}
\end{prop}
 
\begin{proof}

We define
\[
\Lambda_\nu(z) = z^{-\nu} I_\nu(z),
\]
and recall, see \eqref{logder2} below, that
\begin{equation}\label{lambdaprime}
\frac{d}{dz} \log \Lambda_\nu(z) = y_\nu(z), \ \ \ \ \ \  \ \ \ \ \ \ \ z>0.
\end{equation}

Notice that since the right-hand side is strictly positive for $z>0$, the equation \eqref{lambdaprime} says in particular that $\Lambda_\nu$ is strictly increasing on $(0,\infty)$.
If we rewrite \eqref{fs} as
\[
p^{(a)}(z,\zeta,t) = (2t)^{-\frac{a+1}{2}} \Lambda_{\frac{a-1}{2}}\left(\frac{z\zeta}{2t}\right) e^{-\frac{z^2+\zeta^2}{4t}},
\]
then we have
\begin{align*}
 \log p^{(a)}(z,\zeta,t) & = - \frac{a+1}{2} \log(2t) +\log \Lambda_{\frac{a-1}{2}}\left(\frac{z\zeta}{2t}\right) - \frac{z^2+\zeta^2}{4t}.
\end{align*}
The chain rule and \eqref{lambdaprime} give
\begin{equation}\label{derz}
\de_z \log p^{(a)}(z,\zeta,t) = y_{\frac{a-1}{2}}\left(\frac{z\zeta}{2t}\right) \frac{\zeta}{2t} - \frac{z}{2t}.
\end{equation} 
Analogously, we find
\begin{equation}\label{dert}
\de_t \log p^{(a)}(z,\zeta,t) = - \frac{a+1}{2t} -  y_{\frac{a-1}{2}}\left(\frac{z\zeta}{2t}\right) \frac{z\zeta}{2t^2} + \frac{z^2+\zeta^2}{4t^2}.
\end{equation} 
From \eqref{derz} and \eqref{dert} we conclude
\begin{align*}
& \left(\de_z \log p^{(a)}(z,\zeta,t)\right)^2 - \de_t \log p^{(a)}(z,\zeta,t) 
\\
& = y_{\frac{a-1}{2}}\left(\frac{z\zeta}{2t}\right)^2 \frac{\zeta^2}{4t^2} + \frac{z^2}{4t^2} + \frac{a+1}{2t} - \frac{z^2+\zeta^2}{4t^2}
\\
& = \frac{a+1}{2t} + \frac{\zeta^2}{4t^2} \left(y_{\frac{a-1}{2}}\big(\frac{z\zeta}{2t}\big)^2 - 1\right).
\end{align*}
We conclude that \eqref{ly} is valid. 
To establish \eqref{ly22} it suffices to observe that in view of \eqref{zerozero} we obtain for any fixed $\zeta>0$ and $t>0$, as $z\to 0^+$
\[
y_{\frac{a-1}{2}}\left(\frac{z\zeta}{2t}\right) = O\left(\frac{z\zeta}{2t}\right)\ \longrightarrow\ 0.
\]
Since by \eqref{pazero} $z\to p^{(a)}(z,\zeta,t)$ is continuous up to $z = 0$, this shows that
\[
\left(\de_z \log p^{(a)}(0,\zeta,t)\right)^2 - \de_t \log p^{(a)}(0,\zeta,t) = \frac{a+1}{2t} - \frac{\zeta^2}{4t^2} < \frac{a+1}{2t}.
\]

\end{proof}

We show next that, if we restrict the range of $a$, then a global Li-Yau inequality similar to \eqref{ly22} above holds. By this we mean that $z = 0$ can be replaced by \emph{any} $z>0$.

\begin{prop}[Inequality of Li-Yau type for $p^{(a)}(z,\zeta,t)$]\label{C:ly}
Let $a\ge 0$. Then, for every $z, \zeta\in \R^+$ and $t>0$ one has
\begin{equation}\label{ly2}
\left(\de_z \log p^{(a)}(z,\zeta,t)\right)^2 - \de_t \log p^{(a)}(z,\zeta,t) < \frac{a+1}{2t}.
\end{equation} 
\end{prop}

\begin{proof}
Since for $\nu\ge -1$ both $I_\nu(z)$ and $I_{\nu+1}(z)$ are positive for $z>0$, if we let $\nu = \frac{a-1}2$, then since $a \ge 0$ we have $\nu \ge - 1/2$, and in view of Proposition \ref{P:LYp}, proving \eqref{ly2} is equivalent to showing 
\begin{equation*}
y_{\nu}(z) <  1,
\end{equation*}
for every $\nu\ge -\frac 12$, and for every $z>0$. But this follows from Proposition \ref{P:soni} below. 

\end{proof}

\begin{remark}\label{R:si}
One should note the strict inequality in \eqref{ly2} which follows from \eqref{quotients}. It would be interesting to know whether a sharper Li-Yau inequality can be derived by using a sharper upper bound on the function  $y_{\nu}(z)$. In this connection, one should see the papers \cite{Amos}, \cite{SS}, \cite{Se}, \cite{HG1}, \cite{HG2} and \cite{RS}.
\end{remark}

With Proposition \ref{P:LYp} we now return to formula \eqref{repfor} and establish the main result of this section.

\begin{theorem}\label{T:LYgen2}
Let $a\in (-1,\infty)$. Let $\vf\ge 0$ be such that $\vf\in C_{(a)}^1(0,\infty)$. For every $z>0$ and $t>0$ we have the following \emph{adjusted Li-Yau inequality} for the function $P^{(a)}_t \vf(z)$ defined by \eqref{repfor}
\begin{align}\label{LYbessel}
& \left(\de_z \log P^{(a)}_t \vf(z)\right)^2 -  \de_t \log P^{(a)}_t \vf(z) \le  \frac{a+1}{2t}
\\
& + \frac{1}{P^{(a)}_t \vf(z)} \int_0^\infty \vf(\zeta) \frac{\zeta^2}{4t^2} \left(y_{\frac{a-1}{2}}\big(\frac{z\zeta}{2t}\big)^2 - 1\right)  p^{(a)}(z,\zeta,t) \zeta^a d\zeta.
\notag
\end{align}
When $z = 0$, we have for every $t>0$
\begin{align}\label{LYbessel00}
& \left(\de_z \log P^{(a)}_t \vf(0)\right)^2 -  \de_t \log P^{(a)}_t \vf(0) \le  \frac{a+1}{2t}.
\end{align}
When $a\ge 0$ an inequality similar to \eqref{LYbessel00} continues to be valid globally, i.e., for every $z>0$ and $t>0$ one has
\begin{equation}\label{LYbessel2}
\left(\de_z \log P^{(a)}_t \vf(z)\right)^2 -  \de_t \log P^{(a)}_t \vf(z) <  \frac{a+1}{2t}.
\end{equation}
\end{theorem}

\begin{proof}
In what follows we denote for simplicity $u(z,t) = P^{(a)}_t \vf(z)$. Differentiating \eqref{repfor} with respect to $z$ gives
\begin{align}\label{duz}
\de_z u(z,t)^2 & = \left(\int_0^\infty \vf(\zeta) \de_z p^{(a)}(z,\zeta,t) \zeta^a d\zeta\right)^2
\\
& = \left(\int_0^\infty \vf(\zeta) \frac{\de_z p^{(a)}(z,\zeta,t)}{p^{(a)}(z,\zeta,t)^{1/2}} p^{(a)}(z,\zeta,t)^{1/2} \zeta^a d\zeta\right)^2
\notag\\
& \le \int_0^\infty \vf(\zeta) \frac{\de_z p^{(a)}(z,\zeta,t)^2}{p^{(a)}(z,\zeta,t)} \zeta^a d\zeta \int_0^\infty \vf(\zeta) p^{(a)}(z,\zeta,t) \zeta^a d\zeta
\notag\\
& \le u(z,t) \int_0^\infty \vf(\zeta) \frac{\de_z p^{(a)}(z,\zeta,t)^2}{p^{(a)}(z,\zeta,t)} \zeta^a d\zeta,
\notag\end{align}
where in the second to the last inequality we have applied Cauchy-Schwarz. 
In the latter inequality we now substitute \eqref{ly} from Proposition \ref{P:LYp} which we rewrite as follows
\begin{align*}
& \frac{\de_z p^{(a)}(z,\zeta,t)^2}{p^{(a)}(z,\zeta,t)}  = \de_t p^{(a)}(z,\zeta,t) + \frac{a+1}{2t} p^{(a)}(z,\zeta,t) 
\\
& + \frac{\zeta^2}{4t^2} \left(y_{\frac{a-1}{2}}\big(\frac{z\zeta}{2t}\big)^2 - 1\right) p^{(a)}(z,\zeta,t).
\end{align*}
We find
\begin{align*}
& \int_0^\infty \vf(\zeta) \frac{\de_z p^{(a)}(z,\zeta,t)^2}{p^{(a)}(z,\zeta,t)} \zeta^a d\zeta = \int_0^\infty \vf(\zeta) \de_t p^{(a)}(z,\zeta,t) \zeta^a d\zeta
\\
& + \frac{a+1}{2t} \int_0^\infty \vf(\zeta) p^{(a)}(z,\zeta,t) \zeta^a d\zeta 
\\
& + \int_0^\infty \vf(\zeta) \frac{\zeta^2}{4t^2} \left(y_{\frac{a-1}{2}}\big(\frac{z\zeta}{2t}\big)^2 - 1\right)  p^{(a)}(z,\zeta,t) \zeta^a d\zeta
\\
& = \de_t u(z,t) + \frac{a+1}{2t} u(z,t) + \int_0^\infty \vf(\zeta) \frac{\zeta^2}{4t^2} \left(y_{\frac{a-1}{2}}\big(\frac{z\zeta}{2t}\big)^2 - 1\right) p^{(a)}(z,\zeta,t) \zeta^a d\zeta.
\end{align*}
Substituting in \eqref{duz} and dividing by $u(z,t)^2$ in the resulting inequality we conclude that \eqref{LYbessel} does hold. 
If we argue similarly, but use \eqref{ly22} instead of \eqref{ly}, we obtain \eqref{LYbessel00}.

Finally, to establish \eqref{LYbessel2} we use Corollary \ref{C:ly} instead of Proposition \ref{P:LYp}, or simply observe that Proposition \ref{P:soni} guarantees that
\[
\int_0^\infty \vf(\zeta) \frac{\zeta^2}{4t^2} \left(y_{\frac{a-1}{2}}\big(\frac{z\zeta}{2t}\big)^2 - 1\right)  p^{(a)}(z,\zeta,t) \zeta^a d\zeta < 0.
\]

\end{proof}

\begin{remark}\label{R:otherrange}
It is natural to wonder whether, in the range $-1<a<0$, there is a ``good" Li-Yau inequality that can be derived from \eqref{LYbessel}, similarly to what happens for the case of negative Ricci lower bounds in \cite{LY}. It should be clear to the reader that the main obstruction to answering this question in the affirmative is represented by the term
\begin{align*}
w(z,t) & \overset{def}{=} \frac{1}{4t^2} \int_0^\infty \vf(\zeta) \zeta^2 \left(y_{\frac{a-1}{2}}\big(\frac{z\zeta}{2t}\big)^2 - 1\right)  p^{(a)}(z,\zeta,t) \zeta^a d\zeta.
\end{align*}
The difficulty here is created  by the presence of the factor $\zeta^2$ in the integral in the right-hand side. If we had $\zeta$ instead, then we could use Proposition \ref{P:asy} to control $w(z,t)$ in terms of $u(z,t)$. 
\end{remark}
 


With Theorem \ref{T:LYgen2} in hands, we can now establish the scale invariant Harnack inequality for $P^{(a)}_t $ in Theorem \ref{T:harnackbessel0} above. Since the proof is analogous to that of Theorem \ref{T:harext0} in Section \ref{S:sharpext} below, we omit it and refer the reader to that source. 


\section{A comparison with the results of Chiarenza-Serapioni and of Epstein-Mazzeo}\label{S:EM}

Because of its relevance in extension problems, it is interesting to understand what happens with Theorem \ref{T:harnackbessel0} in the remaining range $-1<a<0$. Remarkably, in such  range a Harnack inequality continues to hold.

One way to see this is to observe that a nonnegative solution to the equation
$\p_t u - \Ba u = 0$, which in addition satisfies the condition $\underset{z\to 0^+}{\lim} \paa u(z,t) = 0$ for every $t>0$, must be smooth in $z$ (and in fact, real analytic) up to the vertical line $z = 0$. Therefore, if we set $U(z,t) = u(|z|,t)$ we see that $U$ is a nonnegative solution in the whole half-plane $\R\times (0,\infty)$ to 
\[
\p_t(|z|^a U) - \p_z(|z|^a \p_z U) = 0,
\]
which is a special case of \eqref{csA} above, with $\omega(z) = |z|^a$. Since $\omega\in A_2(\R)$ if and only if $|a|<1$, by Theorem 2.1  \cite{CS} we conclude that a parabolic Harnack inequality holds for $U$ on $\R\times (0,\infty)$. From this, we immediately obtain a Harnack inequality for $u$ up to the vertical line $z = 0$ in the range $-1<a<1$. Two comments are in order though:
\begin{itemize}
\item[(i)] while in Theorem \ref{T:harnackbessel0} we obtain the sharp constant
\begin{equation}\label{sharpc}
\left(\frac ts\right)^{\frac{a+1}2} \exp\left( \frac{(z-\zeta)^2}{4(t-s)}\right),
\end{equation}
that in the Harnack inequality (2.4) in Theorem 2.1 in \cite{CS} is not explicitly known; 
\item[(ii)] Theorem \ref{T:harnackbessel0} holds for any $a\ge 0$, whereas when $a>1$ the results in \cite{CS} are no longer available as the weight $\omega(z) = |z|^a$ is not a $A_2$ weight of Muckenhoupt.
\end{itemize}
 
Another way to see that the Harnack inequality is true also in the range $-1<a<0$ is as follows. In their work \cite{EM} Epstein and Mazzeo studied the diffusion process associated with a class of degenerate parabolic equations from population biology. One the central models of interest for them was the Wright-Fisher operator
\[
\mathscr L_{WF} = x(1-x) \frac{d^2}{dx^2},
\]
which represents the diffusion limit of a Markov chain modeling the
frequency of a gene with $2$ alleles, without mutation or selection. One should also see the seminal paper by Feller \cite{Fe}, the papers \cite{KT} and  \cite{PY},
as well as the paper \cite{CS}, which appeared in the same issue as \cite{EM}, and the more recent works \cite{EM2}, \cite{EM3} and \cite{EP}.

In \cite{EM} and \cite{CS} the authors independently, and with different approaches, construct a parametrix for the Wright-Fisher operator $\mathscr L_{WF}$ by first localizing the analysis to a neighborhood of the boundary points $x = 0$ and $x = 1$. At this point, the approach in \cite{EM} is purely analytical, whereas that in \cite{CS} is more probabilistic. In \cite{EM} the authors by a suitable change of variable are thus led to consider the model operator on $(0,\infty)$
\begin{equation}\label{Lb}
\mathscr L_b = x \frac{\p^2}{\p x^2} + b \frac{\p}{\p x},
\end{equation}
where $b> 0$ is a given number, and they construct the fundamental solution for the Cauchy problem \eqref{cpLb} below. For a given $T>0$ they consider the domain
$D_T = [0,\infty) \times [0,T]$,
and study the problem 
\begin{equation}\label{cpLb}
\begin{cases}
\frac{\p v}{\p t} - \mathscr L_b v = 0\ \ \ \ \ \ \ \text{in}\ \ D_T,
\\
v(x,0) =  f(x),\ \ \ \ \ \ \ \ x>0,
\end{cases}
\end{equation}
under Feller's zero flux condition
\begin{equation}\label{zeroflux}
\underset{x\to 0^+}{\lim}\ x^b \frac{\p v}{\p x}(x,t) = 0.
\end{equation}
Since solutions to \eqref{cpLb} may not be smooth up to $x = 0$, even if the initial datum $f$ is smooth, the condition \eqref{zeroflux} is needed to single out smooth solutions.
Following a suggestion by C. Fefferman, in (6.13) of \cite{EM} the authors find the following representation for the solution of \eqref{cpLb} 
\begin{equation}\label{sLb}
v(x,t) = \int_0^\infty k_b(x,y,t) f(y) dy,
\end{equation}
where (the following is formula (6.14) in \cite{EM})
\begin{equation}\label{kb}
k_b(x,y,t) = \frac 1t \left(\frac xy\right)^{\frac{1-b}2} e^{-\frac{x+y}t} I_{b-1}\left(\frac{2\sqrt{xy}}t\right).
\end{equation}
We recognize next that, via a simple change of variable the Cauchy problem \eqref{cpLb}, \eqref{zeroflux} is the same as \eqref{CP} above.

\begin{prop}\label{P:equiv}
The transformation
\begin{equation}\label{tran}
x = \frac{z^2}4,\ \ \ \ \ \ \ \ \ \ \ b = \frac{a+1}2,
\end{equation}
sends in a one-to-one, onto fashion solutions of \eqref{CP} with $a>-1$ into solutions of \eqref{cpLb} with $b>0$ and with the Neumann condition \eqref{zeroflux}.
As a consequence, the problem \eqref{cpLb} with the boundary condition \eqref{zeroflux} is equivalent to the 
problem \eqref{CP} for the Bessel operator $\Ba$.
\end{prop}

\begin{proof}
To see this, consider a function $v(x,t)$ and define 
\[
u(z,t) = v(z^2/4,t).
\]
Then, the chain rule gives
\[
u_z = v_x\ x_z,\ \ \ u_{zz} = v_{xx}\ (x_z)^2 + v_x\ x_{zz},
\]
and thus
\[
u_t - \Ba u = v_t - v_{xx}\ (x_z)^2 - v_x\left( x_{zz} + \frac az x_z\right).
\]
Now, we have
\[
x_{zz} + \frac az x_z = \frac{a+1}2 \overset{def}{=} b,\ \ \ \ \ \ (x_z)^2 = \frac{z^2}4 = x.
\]
We conclude that 
\begin{equation}\label{BaLb}
u_t(z,t) - \Ba u(z,t) = v_t(x,t) - \mathscr L_b v(x,t).
\end{equation}
Furthermore, one easily verifies that
\begin{equation}\label{nanb}
z^a u_z(z,t) = 2^{2b-1} x^b v_x(x,t).
\end{equation}
The equations \eqref{BaLb}, \eqref{nanb} prove the proposition. 

\end{proof}

 As a consequence of Proposition \ref{P:equiv}, letting $x = z^2/4$ in the left-hand side of \eqref{sLb}, and making the change of variable $y = \zeta^2/4$ in the integral in the right-hand side, we expect the representation formula \eqref{sLb} to become exactly the formula \eqref{repfor} above. This is precisely the case, as the following simple verification shows:
\begin{align*}
u(z,t) & = v(z^2/4,t)  = \int_0^\infty y k_b(z^2/4,y,t) f(y) \frac{dy}y
\\
&  = \frac 12 \int_0^\infty \zeta k_b(z^2/4,\zeta^2/4,t) f(\zeta^2/4) d\zeta
\\
& = \frac 1{2t} \int_0^\infty \zeta k_{\frac{a+1}2}(z^2/4,\zeta^2/4,t) f(\zeta^2/4) d\zeta
\\
& = \frac 1{2t} \int_0^\infty (z\zeta)^{\frac{1-a}2} e^{-\frac{z^2+\zeta^2}{4t}} I_{\frac{a-1}2}\left(\frac{z\zeta}{2t}\right) \vf(\zeta) \zeta^a d\zeta
\\
& = \int_0^\infty \vf(\zeta) p^{(a)}(z,\zeta,t) \zeta^a d\zeta,
\end{align*}
where we have let $\vf(\zeta) = f(\zeta^2/4)$.

In Theorem 4.1 in their paper \cite{EM3} Epstein and Mazzeo by a remarkable adaption of the method of De Giorgi-Nash-Moser, and subsequent contributions of Saloff-Coste and Grigor'yan,  establish a scale invariant Harnack inequality for a large class of degenerate parabolic equations defined on manifolds with corners. The relevant partial differential operators, known as generalized Kimura operators, arise in population biology and they contain as a special case the model \eqref{cpLb} for the full range $b>0$. As a consequence of the results in \cite{EM3}, and of Proposition \ref{P:equiv} above, one obtains a Harnack inequality for positive solutions of \eqref{CP} also in the range $-1<a<0$ which is not covered by our Theorem \ref{T:harnackbessel0}. However, it is not clear to this author that the Harnack inequality in (143) in \cite{EM3} is capable of producing the sharp constant \eqref{sharpc}
in the right-hand side of \eqref{harnackba} above. We also mention the paper \cite{EP} that contains a very different approach to the Harnack inequality, based on a stochastic representation of the solutions, for  more  more general classes of Kimura operators. 


\section{A sharp Harnack inequality for the parabolic extension problem}\label{S:sharpext}

In this section we consider in $\Rnp\times (0,\infty)$ the so-called extension operator for the fractional powers  $(\p_t - \Delta)^s$, $0<s<1$, of the heat operator. Hereafter, for $x\in \Rn$ and $z>0$ we denote by $X = (x,z)\in \Rnp$, and by $(X,t)$ the generic point in $\Rnp\times (0,\infty)$. We also indicate $Y = (y,\zeta) \in \Rnp$ and $(Y,t)$. Given a number $a\in (-1,1)$, the extension operator is the degenerate parabolic operator defined by 
\begin{equation}\label{extop}
\LL u = \p_t(z^a u) - \operatorname{div}_X(z^a \nabla_X u). 
\end{equation}
It was recently introduced independently by Nystr\"om-Sande in \cite{NS}, and Stinga-Torrea in \cite{ST}. These authors proved that if for a given $\vf\in \mathscr S(\Rnn)$, the function $u$ solves the problem 
\[
\begin{cases}
\LL u = 0 \ \ \ \ \ \ \ \ \ \ \ \ \ \text{in}\ \Rnp\times (0,\infty),
\\
u(x,0,t) = \vf(x,t),
\end{cases}
\]
then, with $s\in (0,1)$ determined by the equation $a = 1-2s$, one has
\[
- \frac{2^{-a} \Gamma\left(\frac{1-a}2\right)}{\Gamma\left(\frac{1+a}2\right)} \underset{z\to 0^+}{\lim} z^a \frac{\p u}{\p z}(x,z,t) = (\p_t -\Delta)^s \vf(x,t).
\]
The reader should compare the latter equation with \eqref{dtn} above.

In what follows we establish a remarkable sharp Harnack inequality for the semigroup associated with the operator $\LL$. The first (important) observation is that the Neumann fundamental solution for $\LL$, with singularity at $(Y,0) = (y,\zeta,0)$, is given by
\begin{equation}\label{sGa}
\mathscr G_a(X,Y,t) = p(x,y,t) p^{(a)}(z,\zeta,t),
\end{equation}
where $p(x,y,t) = (4\pi t)^{-n/2} \exp(-\frac{|x-y|^2}{4t})$ is the standard heat kernel in $\Rn\times (0,\infty)$ and $p^{(a)}(z,\zeta,t)$ is given by \eqref{fs} above.

Using Proposition \ref{P:sc}, and the well-known fact that $\int_{\Rn} p(x,y,t) dy = 1$ for every $x\in \Rn$ and $t>0$, it is a trivial exercise to verify that for every $X\in \Rnp$ and $t>0$ one has
\begin{equation}\label{stochcompl}
\int_{\Rnp} \mathscr G_a(X,Y,t) \zeta^a dY = 1.
\end{equation}

Given a function $\vf\in C^\infty_0(\Rnp)$, consider the Cauchy problem with Neumann condition
\begin{equation}\label{cpext}
\begin{cases}
\LL u = 0\ \ \ \ \ \ \ \ \ \ \ \ \ \ \ \ \ \ \text{in}\ \Rnp\times (0,\infty)
\\
u(X,0) = \vf(X),\ \ \ \ \ \ \ X\in \Rnp,
\\
\underset{z\to 0^+}{\lim} \paa u(x,z,t) = 0.
\end{cases}
\end{equation}
The solution to \eqref{cpext} is represented by the formula
\begin{equation}\label{solform}
u(X,t) = \int_{\Rnp} \vf(Y) \mathscr G_a(X,Y,t) \zeta^a dY.
\end{equation}

The following is the main result of this section.

\begin{theorem}\label{T:harext}
Let $a\ge 0$. Let $\vf \ge 0$ be a function for which $u$ given by \eqref{solform} represents a classical solution to \eqref{cpext}. Then, for every $X, Y\in \Rnp$ and every $0<s<t<\infty$, we have
\[
u(X,s) <  u(Y,t) \left(\frac ts\right)^{\frac{n+a+1}2} \exp\left(-\frac{|X-Y|^2}{4(t-s)}\right).
\]
\end{theorem}
  
The proof of Theorem \ref{T:harext} is based on the following inequality of Li-Yau type for $u$.

\begin{theorem}\label{T:LYext}
Let $a\ge 0$ and $\vf$ and $u$ be as in Theorem \ref{T:harext}. Then, for any $X\in \Rnp$ and $t>0$ one has
\[
|\nabla_X \log u(X,t)|^2 - \p_t \log u(X,t) <  \frac{n+a+1}{2t}.
\]
\end{theorem} 

\begin{proof}
Differentiating under the integral sign in \eqref{solform} and applying the Cauchy-Schwarz inequality similarly to the proof of \eqref{duz}, we find
\begin{align}\label{LYuext}
& |\nabla_X  u(X,t)|^2 \le \int_{\Rnp} \vf(Y) \frac{|\nabla_X \mathscr G_a(X,Y,t)|^2}{\mathscr G_a(X,Y,t)} \zeta^a dY \int_{\Rnp} \vf(Y) \mathscr G_a(X,Y,t) \zeta^a dY
\\
& = u(X,t) \int_{\Rnp} \vf(Y) \frac{|\nabla_X \mathscr G_a(X,Y,t)|^2}{\mathscr G_a(X,Y,t)} \zeta^a dY.
\notag
\end{align}
Next, we prove the following crucial result: for every $X, Y\in \Rnp$ and every $t>0$ one has
\begin{equation}\label{LYGa}
|\nabla_X \log \mathscr G_a(X,Y,t)|^2 - \p_t \log \mathscr G_a(X,Y,t) <  \frac{n+a+1}{2t}.
\end{equation}
To establish \eqref{LYGa} we note that from the equation \eqref{sGa} we obtain
\begin{align*}
& |\nabla_X \log \mathscr G_a(X,Y,t)|^2 - \p_t \log \mathscr G_a(X,Y,t) 
= |\nabla_X \log p(x,y,t)|^2 - \p_t \log p(x,y,t)
\\
& + (\p_z \log p^{(a)}(z,\zeta,t))^2 - \p_t \log p^{(a)}(z,\zeta,t).
\end{align*}
The claim \eqref{LYGa} now follows from \eqref{heat} and from Proposition \ref{C:ly} above.

With \eqref{LYGa} in hands, we return to the integral in the right-hand side of \eqref{LYuext} and proceed as follows
\begin{align*}
& \int_{\Rnp} \vf(Y) \frac{|\nabla_X \mathscr G_a(X,Y,t)|^2}{\mathscr G_a(X,Y,t)} \zeta^a dY  < \int_{\Rnp} \vf(Y) \mathscr G_a(X,Y,t) \p_t \log \mathscr G_a(X,Y,t) \zeta^a dY
\\
& +   \frac{n+a+1}{2t} \int_{\Rnp} \vf(Y) \mathscr G_a(X,Y,t) \zeta^a dY
\\
& = \p_t \int_{\Rnp} \vf(Y) \mathscr G_a(X,Y,t) \zeta^a dY + \frac{n+a+1}{2t} u(X,t)
\\
& = \p_t u(X,t) + \frac{n+a+1}{2t} u(X,t).
\end{align*}
We conclude that
\[
|\nabla_X  u(X,t)|^2 < u(X,t) \p_t u(X,t) + \frac{n+a+1}{2t} u(X,t)^2.
\]
Dividing by $u(X,t)^2$ we reach the desired conclusion.

\end{proof}

Having established Theorem \ref{T:LYext0} we finally turn to the 

\begin{proof}[Proof of Theorem \ref{T:harext0}]
The argument is the same as that in \cite{LY}. We repeat it here for the sake of completeness. Fix two points $(X,s), (Y,t)\in \Rnp \times (0,\infty)$, with $0<s<t<\infty$, and consider the straight-line segment (a geodesic line) which starts from $(Y,t)$ and ends in $(X,s)$. We parametrize it by
\[
\alpha(\tau) = (Y + \tau(X-Y), t - \tau(t-s)), \ \ \ \ \ \ \ 0\le \tau \le 1.
\]
Clearly, 
\[
\alpha'(\tau) = (X-Y,- (t-s)), \ \ \ \ \ \ \ 0\le \tau \le 1.
\]
We now consider the function
\[
h(\tau) = \log U(\alpha(\tau)).
\]
We have
\begin{align*}
& \log \frac{U(X,s)}{U(Y,t)} = h(1) - h(0) = \int_0^1 h'(\tau) d\tau 
\\
& = \int_0^1 <\nabla_X \log U(\alpha(\tau)),X-Y> d\tau - (t-s) \int_0^1\p_t \log U(\alpha(\tau)) d\tau
\\
& \le |X-Y| \int_0^1 |\nabla_X \log U(\alpha(\tau))| d\tau - (t-s) \int_0^1\p_t \log U(\alpha(\tau)) d\tau
\\
& \le \frac{|X-Y|^2}{2\ve} + \frac{\ve}2 \int_0^1 |\nabla_X \log U(\alpha(\tau))|^2 d\tau - (t-s) \int_0^1\p_t \log U(\alpha(\tau)) d\tau.
\end{align*}
At this point we observe that Theorem \ref{T:LYext0} implies for every $0\le \tau \le 1$
\begin{align*}
- \de_t \log u(\alpha(\tau))  & < \frac{n+ a+1}{2(t -\tau(t-s))} - |\nabla_X \log U(\alpha(\tau))|^2.
\end{align*}
Replacing this information in the above inequality, we find
\begin{align*}
& \log \frac{U(X,s)}{U(Y,t)} \le  \frac{|X-Y|^2}{2\ve} + \frac{\ve}2 \int_0^1 |\nabla_X \log U(\alpha(\tau))|^2 d\tau
\\
& + \log \left(\frac ts\right)^{\frac{n+a+1}{2}} - (t-s) \int_0^1 |\nabla_X \log U(\alpha(\tau))|^2 d\tau.
\end{align*}
If we now choose $\ve>0$ such that $\ve = 2(t-s)$, we finally obtain
\[
\log \frac{U(X,s)}{U(Y,t)} \le  \frac{|X-Y|^2}{4(t-s)} + \log \left(\frac ts\right)^{\frac{n+a+1}{2}}.
\]
Exponentiating both sides of this inequality we reach the desired conclusion.

\end{proof}


\section{Monotonicity formulas of Struwe and Almgren-Poon type for the Bessel semigroup}\label{S:mono}

In this section we prove Theorems \ref{T:struwe0} and \ref{T:poon0}. As we have pointed out, these results provide one more interesting instance of the underlying theme of this paper. As it will be apparent from the proofs, remarkably this time is not the inequality \eqref{soni0} that lurks in the shadows, but rather the stronger monotonicity property of the Bessel quotient $y_\nu = I_{\nu+1}/I_\nu$ in Proposition \ref{P:monotonicityy} below.   

In $\mathscr Q = (0,\infty)\times (0,\infty)$ we consider a solution $u$ of the heat equation
\begin{equation}\label{eu}
\p_t u  - \Bz u = 0,
\end{equation}
which for every $t>0$ satisfies the Neumann condition
\begin{equation}\label{snc}
\underset{\zeta\to 0^+}{\lim} \paz u(\zeta,t) = 0.
\end{equation}
For a given $z>0$ and $T>0$ we introduce the following scaled energy, \emph{centered at} $(z,T)$, with respect to the backward Gaussian-Bessel measure 
\begin{equation}\label{e1}
E^{(a)}_{z,T}(t) \overset{def}{=} \frac{T-t}2 \int_0^\infty (\p_\zeta u(\zeta,t))^2 p^{(a)}(z,\zeta,T-t) \zeta^a d\zeta.
\end{equation}
It is obvious that without further assumptions it is not guaranteed that the integral \eqref{e1}, as well as those that will appear in the proofs of this section, be convergent. This difficulty is serious and in order to circumvent it one needs to:
\begin{itemize}
\item[(a)] multiply the function $u$ by a suitable cutoff function as it was done in the monograph \cite{DGPT} in the study of closely related monotonicity properties. Since doing this changes the equation satisfied by $u$, the analysis becomes considerably more complicated;
\item[(b)] develop the regularity theory which is necessary to rigorously justify all integration by parts that occur when differentiating \eqref{e1}. 
\end{itemize}

Since our intent is to point to a new phenomenon, we will ``wave our hands" on these important aspects, and refer the reader to \cite{DGPT}, \cite{BG} and \cite{BDGP} for a rigorous treatment. As a consequence we will from now on assume that \emph{all} integrations  by parts are justifed, and \emph{all} boundary terms vanish.

We note explicitly that, in the $\G$-language of Section \ref{S:bessel}, the above energy can be written (see \eqref{gga} above)
\[
E^{(a)}_{z,T}(t) = \frac{T-t}2 P^{(a)}_{T-t}(\G^{(a)}(u))(z).
\]
In the proof of Theorem \ref{T:struwe0} below, we will also distinguish between the cases $z>0$ and $z = 0$. When $z = 0$ in the integral in \eqref{e1} above we integrate against $p^{(a)}(0,\zeta,T-t)$ (for the value of this function see \eqref{pazero} above). For simplicity, in this case we will write $E_T(t)$ instead of $E_{0,T}(t)$.

For later use, we also observe that the function $g^{(a)}(\zeta,t) = p^{(a)}(z,\zeta,T-t)$ satisfies the backward heat equation
\begin{equation}\label{backheat}
\p_t g^{(a)}(\zeta,t) = - \Bz g^{(a)}(\zeta,t).
\end{equation}

Henceforth, to simplify the notation we will indicate partial derivatives with $u_\zeta, u_{\zeta \zeta}, u_t$, etc. We will also routinely drop the arguments of all functions appearing in the integral in \eqref{e1}, and write $u_\zeta, p^{(a)}, p^{(a)}_\zeta$, etc., instead of $u_\zeta(\zeta,t), p^{(a)}(z,\zeta,T-t), p^{(a)}_\zeta(z,\zeta,T-t)$, etc.

\begin{proof}[Proof of Theorem \ref{T:struwe0}]

Differentiating \eqref{e1} we find
\begin{align*}
\frac{d E^{(a)}_{z,T}}{dt}(t) & = \frac{T-t}2 \int_0^\infty -\frac{1}{2(T-t)} u_\zeta^2 p^{(a)} \zeta^a d\zeta + (T-t) \int_0^\infty u_\zeta u_{t\zeta} p^{(a)} \zeta^a d\zeta
\\
& - \frac{T-t}2 \int_0^\infty u_\zeta^2 \Bz p^{(a)} \zeta^a d\zeta,
\end{align*}
where in the last term we have used \eqref{backheat}. We now integrate by parts in the second integral in the right-hand side, obtaining
\begin{align*}
& \int_0^\infty u_\zeta u_{t\zeta} p^{(a)} \zeta^a d\zeta = \zeta^a u_\zeta u_t p^{(a)}\bigg|_{\zeta = 0}^\infty - \int_0^\infty u_{\zeta \zeta} u_{t} p^{(a)} \zeta^a d\zeta - \int_0^\infty \frac a\zeta u_{\zeta} u_{t} p^{(a)} \zeta^a d\zeta
\\
& - \int_0^\infty u_{\zeta} u_{t} p^{(a)}_\zeta \zeta^a d\zeta  = - \int_0^\infty \Bz u\ u_{t} p^{(a)} \zeta^a d\zeta - \int_0^\infty u_{\zeta} u_{t} p^{(a)}_\zeta \zeta^a d\zeta
\\
& =  - \int_0^\infty u_{t}^2  p^{(a)} \zeta^a d\zeta - \int_0^\infty u_{t} u_{\zeta}  p^{(a)}_\zeta \zeta^a d\zeta,
\end{align*}
where in the last equality we have used \eqref{eu}. Substituting in the above identity, we find
\begin{align*}
\frac{d E^{(a)}_{z,T}}{dt}(t) & =  \frac{T-t}2 \int_0^\infty -\frac{1}{2(T-t)} u_\zeta^2 p^{(a)} \zeta^a d\zeta - \frac{T-t}2 \int_0^\infty u_\zeta^2 \Bz p^{(a)} \zeta^a d\zeta
\\ 
& - (T-t) \int_0^\infty u_{t}^2  p^{(a)} \zeta^a d\zeta - (T-t) \int_0^\infty u_{t} u_{\zeta}  p^{(a)}_\zeta \zeta^a d\zeta
\\
& = - (T-t) \int_0^\infty \bigg[u_{t} + u_\zeta \frac{p^{(a)}_\zeta}{p^{(a)}}\bigg]^2 p^{(a)} \zeta^a d\zeta + (T-t) \int_0^\infty u_{t} u_{\zeta}  p^{(a)}_\zeta \zeta^a d\zeta
\\
& + (T-t) \int_0^\infty u_\zeta^2 \frac{(p^{(a)}_\zeta)^2}{(p^{(a)})^2} p^{(a)} \zeta^a d\zeta  + \frac{T-t}2 \int_0^\infty -\frac{1}{2(T-t)} u_\zeta^2 p^{(a)} \zeta^a d\zeta
\\
& - \frac{T-t}2 \int_0^\infty u_\zeta^2 \Bz p^{(a)} \zeta^a d\zeta.
\end{align*}
We finally write
\begin{equation}\label{e2}
\frac{d E^{(a)}_{z,T}}{dt}(t) = - (T-t) \int_0^\infty \bigg[u_{t} + u_\zeta \frac{p^{(a)}_\zeta}{p^{(a)}}\bigg]^2 p^{(a)} \zeta^a d\zeta + G_{z,T}^{(a)}(t),
\end{equation}
where we have set
\begin{align}\label{G}
G_{z,T}^{(a)}(t) & = (T-t) \int_0^\infty u_\zeta^2 \bigg[\frac{(p^{(a)}_\zeta)^2}{(p^{(a)})^2} - \frac{1}{2(T-t)} - \frac{\Bz p^{(a)}}{2p^{(a)}}\bigg] p^{(a)} \zeta^a d\zeta
\\
& + (T-t) \int_0^\infty u_{t} u_{\zeta}  p^{(a)}_\zeta \zeta^a d\zeta.
\notag
\end{align}
Now, we integrate by parts in the second integral in the right-hand side of \eqref{G}, obtaining
\begin{align*}
& \int_0^\infty u_{t} u_{\zeta}  p^{(a)}_\zeta \zeta^a d\zeta = \int_0^\infty \Bz u\ u_{\zeta}  p^{(a)}_\zeta \zeta^a d\zeta = \int_0^\infty u_{\zeta\zeta} u_{\zeta}  p^{(a)}_\zeta \zeta^a d\zeta
\\
& + \int_0^\infty \frac a\zeta u_{\zeta} u_{\zeta}  p^{(a)}_\zeta \zeta^a d\zeta = \int_0^\infty \p_\zeta \left(\frac{u_{\zeta}^2}2\right)   p^{(a)}_\zeta \zeta^a d\zeta + \int_0^\infty \frac a\zeta u_{\zeta}^2  p^{(a)}_\zeta \zeta^a d\zeta
\\
& = \zeta^a p^{(a)}_\zeta \left(\frac{u_{\zeta}^2}2\right)\bigg|_{\zeta = 0}^\infty - \int_0^\infty \frac{u_{\zeta}^2}2   p^{(a)}_{\zeta\zeta} \zeta^a d\zeta - \int_0^\infty \frac a\zeta \frac{u_{\zeta}^2}2   p^{(a)}_\zeta \zeta^a d\zeta
\\
& + \int_0^\infty \frac a\zeta u_{\zeta}^2  p^{(a)}_\zeta \zeta^a d\zeta = \int_0^\infty \frac{u_{\zeta}^2}2  \frac a\zeta  p^{(a)}_\zeta \zeta^a d\zeta - \int_0^\infty \frac{u_{\zeta}^2}2   p^{(a)}_{\zeta\zeta} \zeta^a d\zeta.
\end{align*}
Substituting this result in \eqref{G}, we find
\begin{align}\label{G2}
G_{z,T}^{(a)}(t) & = (T-t) \int_0^\infty u_\zeta^2 \left[\frac{(p^{(a)}_\zeta)^2}{(p^{(a)})^2} - \frac{1}{2(T-t)} - \frac{\Bz p^{(a)}}{2p^{(a)}}\right] p^{(a)} \zeta^a d\zeta
\\
& + (T-t) \int_0^\infty  \frac{u_{\zeta}^2}2  \frac a\zeta p^{(a)}_\zeta \zeta^a d\zeta - (T-t) \int_0^\infty \frac{u_{\zeta}^2}2   p^{(a)}_{\zeta\zeta} \zeta^a d\zeta
\notag
\\
& = - (T-t) \int_0^\infty u_\zeta^2 \left[\frac{p^{(a)}_{\zeta\zeta}}{p^{(a)}} - \frac{(p^{(a)}_\zeta)^2}{(p^{(a)})^2} + \frac{1}{2(T-t)}\right] p^{(a)} \zeta^a d\zeta.
\notag
\end{align}

From \eqref{e2} we see that the proof of the theorem will be completed if we establish the following 

\medskip

\noindent \textbf{Claim:} for every $z >0$ and $0<t<T$, we have when $a\ge 0$
\begin{equation}\label{claim}
G_{z,T}^{(a)}(t) < 0.
\end{equation}
To prove \eqref{claim} it suffices to show that for every $z, \zeta >0$ and $0<t<T$, one has
\begin{equation}\label{claim2}
\frac{p^{(a)}_{\zeta\zeta}}{p^{(a)}} - \frac{(p^{(a)}_\zeta)^2}{(p^{(a)})^2} + \frac{1}{2(T-t)} > 0,
\end{equation}
and we thus turn to proving \eqref{claim2}. Now, \eqref{derz} gives
\begin{align}\label{eta}
\frac{(p^{(a)}_\zeta)^2}{(p^{(a)})^2} & = \left(y_{\frac{a-1}{2}}\left(\frac{z\zeta}{2(T-t)}\right) \frac{z}{2(T-t)} - \frac{\zeta}{2(T-t)}\right)^2.
\end{align} 
Similarly, we obtain
\begin{align*}
p^{(a)}_{\zeta\zeta} & = \bigg(y_{\frac{a-1}{2}}'\bigg(\frac{z\zeta}{2(T-t)}\bigg) \frac{z^2}{4(T-t)^2} - \frac{1}{2(T-t)}\bigg) p^{(a)} 
\\
& + \bigg(y_{\frac{a-1}{2}}\bigg(\frac{z\zeta}{2(T-t)}\bigg) \frac{z}{2(T-t)} - \frac{\zeta}{2(T-t)}\bigg)^2 p^{(a)}.
\end{align*}
From this formula and \eqref{eta}, we find
\begin{align*}
& \frac{p^{(a)}_{\zeta\zeta}}{2p^{(a)}} - \frac{(p^{(a)}_\zeta)^2}{(p^{(a)})^2} + \frac{1}{2(T-t)} = y'_{\frac{a-1}{2}}\left(\frac{z\zeta}{2(T-t)}\right) \frac{z^2}{4(T-t)^2}.
\end{align*}
It is at this point that the monotonicity of the Bessel quotient enters the stage. Invoking the crucial Proposition \ref{P:monotonicityy} we now see that, when $a\ge 0$ (or, equivalently, $\nu \ge - 1/2$), we have $y_{\frac{a-1}{2}}'\left(\frac{z\zeta}{2(T-t)}\right) > 0$ for any $z, \zeta>0$ and every $0<t <T$. This proves \eqref{claim2}, and therefore \eqref{claim}, thus completing the proof of the first part of Theorem \ref{T:struwe0}.

As for the second part, suppose that $z = 0$. Then, for any $a>-1$ we have from \eqref{pazero} above
\[
p^{(a)}(0,\zeta,T-t) = \frac{1}{2^{a} \G(\frac{a+1}2)} (T-t)^{-\frac{a+1}{2}} e^{-\frac{\zeta^2}{4(T-t)}}.
\] 
This gives
\begin{equation}\label{derpa}
p^{(a)}_\zeta = - \frac{\zeta}{2(T-t)} p^{(a)}, \ \ \ \ p^{(a)}_{\zeta\zeta} = \bigg(- \frac{1}{2(T-t)} + \frac{\zeta^2}{4(T-t)^2}\bigg)  p^{(a)}.
\end{equation}
From these formulas we immediately obtain
\[
\frac{p^{(a)}_{\zeta\zeta}}{p^{(a)}} - \frac{(p^{(a)}_\zeta)^2}{(p^{(a)})^2} + \frac{1}{2(T-t)} \equiv 0.
\]
In view of \eqref{G2} we conclude that $G^{(a)}_T \equiv 0$. Substitution in \eqref{e2} finally gives
\[
\frac{d E^{(a)}_{z,T}}{dt}(t) = - (T-t) \int_0^\infty \bigg[u_{t} + u_\zeta \frac{p^{(a)}_\zeta}{p^{(a)}}\bigg]^2 p^{(a)} \zeta^a d\zeta \le 0.
\]

\end{proof}

Next, we show that the Bessel semigroup satisfies a monotonicity property analogous to that proved by Poon for the standard heat equation in \cite{Po}. On a solution $u$ of \eqref{eu} satisfying \eqref{snc}, we introduce the quantity
\begin{equation}\label{L}
L_{z,T}^{(a)}(t) = \frac 12 \int_0^\infty u^2(\zeta,t) p^{(a)}(z,\zeta,T-t) \zeta^a d\zeta,
\end{equation}
where $z>0$ is fixed. We have the following result.

\begin{prop}\label{P:L'}
For every $0<t<T$ one has
\[
\frac{dL_{z,T}^{(a)}}{dt}(t) = \int_0^\infty u \bigg[u_{t} + u_\zeta \frac{p^{(a)}_\zeta}{p^{(a)}}\bigg] p^{(a)} \zeta^a d\zeta.
\]
\end{prop}

\begin{proof}

Differentiating \eqref{L} and using \eqref{backheat}, we find
\begin{align*}
\frac{dL_{z,T}^{(a)}}{dt}(t) & =  \int_0^\infty  u u_t p^{(a)} \zeta^a d\zeta - \frac 12 \int_0^\infty  u^2 \Bz p^{(a)} \zeta^a d\zeta
\\
& = \int_0^\infty  u u_t p^{(a)} \zeta^a d\zeta - \frac 12 \int_0^\infty  u^2 p^{(a)}_{\zeta\zeta} \zeta^a d\zeta - \frac 12 \int_0^\infty  u^2 \frac a\zeta p^{(a)}_\zeta \zeta^a d\zeta.
\end{align*}
Next, we integrate by parts in the second integral in the right-hand side of the latter identity, obtaining
\begin{align*}
& - \frac 12 \int_0^\infty  u^2 p^{(a)}_{\zeta\zeta} \zeta^a d\zeta
= - \frac 12 u^2 \zeta^a p^{(a)}_{\zeta} \bigg|_0^\infty + \int_0^\infty u u_\zeta p^{(a)}_{\zeta} \zeta^a d\zeta + \frac 12 \int_0^\infty  u^2 \frac a\zeta p^{(a)}_{\zeta} \zeta^a d\zeta
\end{align*}
Substituting in the above equation we find
\begin{align*}
\frac{dL_{z,T}^{(a)}}{dt}(t) & = \int_0^\infty  u u_t p^{(a)} \zeta^a d\zeta + \int_0^\infty u u_\zeta p^{(a)}_{\zeta} \zeta^a d\zeta = \int_0^\infty u \bigg[u_{t} + u_\zeta \frac{p^{(a)}_\zeta}{p^{(a)}}\bigg] p^{(a)} \zeta^a d\zeta,
\end{align*}
which is the sought for conclusion.

\end{proof}

We next establish an alternative way of interpreting the function $E_{z,T}^{(a)}(t)$ in \eqref{e1} above.

\begin{lemma}\label{L:alt}
For every $t\in (0,T)$ one has
\[
E_{z,T}^{(a)}(t) = - \frac{T-t}2 \int_0^\infty u \bigg[u_{t} + u_\zeta \frac{p^{(a)}_\zeta}{p^{(a)}}\bigg] p^{(a)} \zeta^a d\zeta.
\]
\end{lemma}

\begin{proof}
We have from \eqref{e1}
\begin{align*}\label{e1}
E^{(a)}_{z,T}(t) & = \frac{T-t}2 \int_0^\infty u_\zeta u_\zeta p^{(a)} \zeta^a d\zeta = \frac{T-t}2 \zeta^a u_\zeta u p^{(a)}\bigg|_{\zeta = 0}^\infty - \frac{T-t}2 \int_0^\infty u_{\zeta\zeta} u p^{(a)} \zeta^a d\zeta 
\\
& - \frac{T-t}2 \int_0^\infty u_\zeta u p^{(a)}_\zeta \zeta^a d\zeta - \frac{T-t}2 \int_0^\infty \frac a\zeta u_\zeta u p^{(a)} \zeta^a d\zeta
\\
& = - \frac{T-t}2 \int_0^\infty u \Bz u\ p^{(a)} \zeta^a d\zeta - \frac{T-t}2 \int_0^\infty u u_\zeta  p^{(a)}_\zeta \zeta^a d\zeta
\\
& = - \frac{T-t}2 \int_0^\infty u \bigg[u_{t} + u_\zeta \frac{p^{(a)}_\zeta}{p^{(a)}}\bigg] p^{(a)} \zeta^a d\zeta,
\end{align*}
where in the last equality we have used \eqref{eu}. This proves the lemma.

\end{proof}

Now, we let $T = 0$, and for fixed $z>0$ and any $r>0$, we consider the two quantities
\begin{equation}\label{HI}
H_{z}^{(a)}(r) \overset{def}{=} L_{z,0}^{(a)}(-r^2) = L_{z}^{(a)}(-r^2),\ \ \ \ \ \ \ \ \ \ \ I_z^{(a)}(r) \overset{def}{=} E_{z,0}^{(a)}(-r^2) = E_{z}^{(a)}(-r^2).
\end{equation}
We note that from \eqref{L} and \eqref{e1} we have
\[
H_z^{(a)}(r) = \frac 12 \int_0^\infty u^2(\zeta,t) p^{(a)}(z,\zeta,r^2) \zeta^a d\zeta,
\]
and
\[
I_z^{(a)}(r)  = \frac{r^2}2 \int_0^\infty u_\zeta^2(\zeta,t) p^{(a)}(z,\zeta,r^2) \zeta^a d\zeta.
\]

\begin{definition}\label{D:frequency}
We define the \emph{frequency centered at $z$} of a solution $u$ to \eqref{eu}, satisfying \eqref{snc}, by the following equation:
\[
N_z^{(a)}(r) = \frac{I_z^{(a)}(r)}{H_z^{(a)}(r)},\ \ \ \ \ \ \ \ \ \ \ \ 0<r<\sqrt T.
\]
When $z = 0$, we simply write $N^{(a)}(r)$, instead of $N_0^{(a)}(r)$. 
\end{definition}

We next establish the second main result of this section.

\begin{proof}[Proof of Theorem \ref{T:poon0}]
Suppose first that $z>0$. Using Definition \ref{D:frequency} we see that the logarithmic derivative of $N_z^a$ is given by
\begin{align}\label{logN}
& \frac{d}{dr} \log N_z^{(a)}(r) = \frac{\frac{d I_z^{(a)}}{dr}(r)}{I_z^{(a)}(r)} - \frac{\frac{d H_z^{(a)}}{dr}(r)}{H_z^{(a)}(r)}.
\end{align}
By \eqref{HI} we have
\begin{equation}\label{H'I'}
\frac{d H_z^{(a)}}{dr}(r) = - 2r \frac{d L_z^{(a)}}{dt}(-r^2),\ \ \ \ \ \ \ \ \ \ \ 
\frac{d I_z^{(a)}}{dr}(r) = - 2r \frac{d E_z^{(a)}}{dt}(-r^2).
\end{equation}
Combining the first identity in \eqref{H'I'} with  Proposition \ref{P:L'}, we have
\begin{equation}\label{HH}
\frac{\frac{d H_z^{(a)}}{dr}(r)}{H_z^{(a)}(r)} = \frac{- 2r \int_0^\infty u \bigg[u_{t} + u_\zeta \frac{p^{(a)}_\zeta}{p^{(a)}}\bigg] p^{(a)} \zeta^a d\zeta}{\frac 12 \int_0^\infty u^2 p^{(a)} \zeta^a d\zeta}.
\end{equation}
From the second identity in \eqref{H'I'} and \eqref{e2}, we now find
\begin{align*}
& \frac{d I_z^{(a)}}{dr}(r)  = 2r^3  \int_0^\infty \bigg[u_{t} + u_\zeta \frac{p^{(a)}_\zeta}{p^{(a)}}\bigg]^2 p^{(a)} \zeta^a d\zeta - 2r G_z^{(a)}(-r^2),
\end{align*}
where $G_z^{(a)} = G_{z,0}^{(a)}$ is given by \eqref{G2}. This equation and Lemma \ref{L:alt} give
\begin{align*}
& \frac{\frac{d I_z^{(a)}}{dr}(r)}{I_z^{(a)}(r)} = \frac{2r^3  \int_0^\infty \bigg[u_{t} + u_\zeta \frac{p^{(a)}_\zeta}{p^{(a)}}\bigg]^2 p^{(a)} \zeta^a d\zeta}{- \frac{r^2}2 \int_0^\infty u \bigg[u_{t} + u_\zeta \frac{p^{(a)}_\zeta}{p^{(a)}}\bigg] p^{(a)} \zeta^a d\zeta} - \frac{2r G_z^{(a)}(-r^2)}{\frac{r^2}2 \int_0^\infty u_\zeta^2(\zeta,t) p^{(a)}(z,\zeta,r^2) \zeta^a d\zeta}.
\end{align*}
We emphasize in the right-hand side of the latter equation we have used Lemma \ref{L:alt} to express $E_z^{(a)}$ in the denominator of the first quotient, whereas we have used the definition \eqref{e1} to write the same quantity in the denominator of the second quotient. It is crucial that there is a minus sign in front of the second quotient. Since the denominator is positive, in view of \eqref{claim} above we can thus conclude that, given $z>0$, then for every $r>0$ we have
\begin{equation}\label{winning}
\frac{\frac{d I_z^{(a)}}{dr}(r)}{I_z^{(a)}(r)} > \frac{4r  \int_0^\infty \bigg[u_{t} + u_\zeta \frac{p^{(a)}_\zeta}{p^{(a)}}\bigg]^2 p^{(a)} \zeta^a d\zeta}{-  \int_0^\infty u \bigg[u_{t} + u_\zeta \frac{p^{(a)}_\zeta}{p^{(a)}}\bigg] p^{(a)} \zeta^a d\zeta}, \  \ \ \ \ \ \ \ \ a \ge 0. 
\end{equation}
Moreover, since in the end of the proof of Theorem \ref{T:struwe0} we showed that, when $z = 0$, then $G^{(a)} \equiv 0$ for every $a>-1$, we also infer that
\begin{equation}\label{winning2}
\frac{\frac{d I^{(a)}}{dr}(r)}{I^{(a)}(r)} = \frac{4 r  \int_0^\infty \bigg[u_{t} + u_\zeta \frac{p^{(a)}_\zeta}{p^{(a)}}\bigg]^2 p^{(a)} \zeta^a d\zeta}{- \int_0^\infty u \bigg[u_{t} + u_\zeta \frac{p^{(a)}_\zeta}{p^{(a)}}\bigg] p^{(a)} \zeta^a d\zeta}, \  \ \ \ \ \ \ \ \ a > -1. 
\end{equation}
Combining \eqref{logN}, \eqref{HH}, \eqref{winning} and \eqref{winning2}, we finally conclude that for every $z>0$ and $r>0$, we have for every $a\ge 0$,
\begin{align}\label{logN2}
\frac{d}{dr} \log N_z^{(a)}(r) & > \frac{4 r  \int_0^\infty \bigg[u_{t} + u_\zeta \frac{p^{(a)}_\zeta}{p^{(a)}}\bigg]^2 p^{(a)} \zeta^a d\zeta}{- \int_0^\infty u \bigg[u_{t} + u_\zeta \frac{p^{(a)}_\zeta}{p^{(a)}}\bigg] p^{(a)} \zeta^a d\zeta}
\\
& - \frac{- 4r \int_0^\infty u \bigg[u_{t} + u_\zeta \frac{p^{(a)}_\zeta}{p^{(a)}}\bigg] p^{(a)} \zeta^a d\zeta}{\int_0^\infty u^2 p^{(a)} \zeta^a d\zeta}.
\notag
\end{align}

If instead $z = 0$, then we obtain for every $a>-1$
\begin{align}\label{logN3}
\frac{d}{dr} \log N^{(a)}(r) & = \frac{4 r  \int_0^\infty \bigg[u_{t} + u_\zeta \frac{p^{(a)}_\zeta}{p^{(a)}}\bigg]^2 p^{(a)} \zeta^a d\zeta}{-  \int_0^\infty u \bigg[u_{t} + u_\zeta \frac{p^{(a)}_\zeta}{p^{(a)}}\bigg] p^{(a)} \zeta^a d\zeta}
\\
& - \frac{- 4 r \int_0^\infty u \bigg[u_{t} + u_\zeta \frac{p^{(a)}_\zeta}{p^{(a)}}\bigg] p^{(a)} \zeta^a d\zeta}{\int_0^\infty u^2 p^{(a)} \zeta^a d\zeta}.
\notag
\end{align}
Since by the inequality of Cauchy-Schwarz the right-hand side in \eqref{logN2}, \eqref{logN3} is always $\ge 0$, we finally obtain
\begin{equation}\label{logNV}
\frac{d}{dr} \log N^{(a)}(r) \begin{cases} >0,\ \ \ \ \ \ \text{if}\ z>0,\ \ \ a\ge 0,
\\
 \ge 0,\ \ \ \ \ \ \text{if}\ z=0,\ \ \ a>- 1.
 \end{cases}
 \end{equation}
This almost concludes the proof of the theorem. We are left with showing the last part, i.e., the characterization of the case of constant frequency when $z = 0$. If for $\kappa > 0$ we have $u(\la \zeta,\la^2 t) = \la^\kappa u(\zeta,t)$ for every $\la>0$, differentiating with respect to $\la$ and setting $\la = 1$, we find $\zeta u_\zeta + 2 t u_t = \kappa u$. If we denote by $Z = Z(\zeta,t) = \zeta \p_\zeta + 2t \p_t$ the generator of the parabolic dilations, then we can write the latter equation as $Zu = \kappa u$. Now, 
\[
Zu(\zeta,-r^2) = \zeta u_\zeta - 2 r^2 u_t.
\]
On the other hand, with $p^{(a)} = p^{(a)}(0,\zeta,r^2)$, we find from \eqref{derpa}
\[
p^{(a)}_\zeta = - \frac{\zeta}{2r^2} p^{(a)}.
\]
Therefore, 
\begin{equation}\label{Zur2}
u_t + u_\zeta \frac{p^{(a)}_\zeta}{p^{(a)}} = u_t - \frac{\zeta}{2 r^2}u_\zeta = \frac{Zu(\zeta,-r^2)}{2(-r^2)}
\end{equation}
Since by the hypothesis we have $Zu(\zeta,-r^2) = \kappa u(\zeta,-r^2)$, we find that
\begin{align*}
I^{(a)}(r) & = - \frac{r^2}{2} \int_0^\infty \frac{u(\zeta,-r^2) Zu(\zeta,-r^2)}{- 2r^2} p^{(a)} \zeta^a d\zeta
\\
& = \frac \kappa4 \int_0^\infty u^2 p^{(a)} \zeta^a d\zeta = \frac \kappa2 H^{(a)}(r).
\end{align*}
In conclusion, if $u$ is homogeneous of degree $\kappa$, then $N^{(a)}(r) \equiv \frac{\kappa}2$.

Vice-versa, suppose that $N^{(a)}(r) \equiv \frac{\kappa}2$. We want to show that $u$ must be parabolically homogeneous of degree $\kappa$, i.e., $Zu = \kappa u$. From our assumption, we have $\frac{d}{dr} \log N^{(a)}(r) \equiv 0$. If $z = 0$,  then by the second option in \eqref{logNV} we obtain that for any $a>-1$ there must be equality in the Cauchy-Schwarz inequality in \eqref{logN3}. This means that for every $r>0$ there exists $\alpha(r)\in \R$ such that for every $\zeta>0$
\[
\bigg[u_{t} + u_\zeta \frac{p^{(a)}_\zeta}{p^{(a)}}\bigg](\zeta,-r^2) = \alpha(r) u(\zeta,-r^2).
\]
From \eqref{Zur2} we can rewrite this equation as
\begin{equation}\label{zum}
\frac{Zu(\zeta,-r^2)}{2(-r^2)} = \alpha(r) u(\zeta,-r^2).
\end{equation}
This information implies that
\[
\frac{I^{(a)}(r)}{H^{(a)}(r)} = - r^2 \alpha(r).
\]
We conclude that $- r^2 \alpha(r) \equiv \kappa/2$, hence 
\[
\alpha(r) \equiv - \frac{\kappa}{2r^2}.
\]
From \eqref{zum} we finally have that for every $\zeta, r>0$ it must be $Zu(\zeta,-r^2) = \kappa u(\zeta,-r^2)$, which is the sought for conclusion.

\end{proof}

In closing we note that the frequency centered at any point $z>0$ does not detect homogeneity.  One reason for this is that the critical information \eqref{Zur2} above is no longer available when the pole of the Bessel-Gaussian $p^{(a)}(z,\zeta,r^2)$ is at a point $z>0$. A second (related) obstruction is in the fact that from the first option in \eqref{logNV} we obtain for $r>0$ the strict inequality
\[
\frac{d}{dr} \log N^{(a)}(r) >0.
\]


\section{Appendix: the modified Bessel function $I_\nu(z)$}\label{S:appendix}

In this section we collect some facts concerning the modified Bessel function $I_\nu$ which are probably well-known to the special functions community, but not so to people working in pde's. For the reader's convenience we also provide a proof of Propositions \ref{P:soni} and \ref{P:monotonicityy}, which are crucial to our work. 

We recall that  the modified Bessel function of the first kind and order $\nu\in \mathbb C$ is defined  by the series
\begin{equation}\label{I}
I_\nu(z) = \sum_{k=0}^\infty \frac{(z/2)^{\nu+2k}}{\G(k+1) \G(k+\nu+1)},\ \ \ \ \ \ \ |z|<\infty,\ |\arg z| < \pi,
\end{equation}
in the complex plane cut along the negative real axis.
From \eqref{I} we see that when $\nu>-1$, then $\G(k+\nu+1)>0$ for any $k\in \mathbb N\cup\{0\}$, and therefore $I_\nu(z)>0$ when $z>0$. The same is true also when $\nu = -1$. 

From \eqref{I} the behavior of $I_\nu(z)$ for small values of $z>0$ easily follows, and we have
\begin{equation}\label{smallz}
\underset{z\to 0^+}{\lim} 2^\nu \G(\nu+1) z^{-\nu} I_\nu(z) = 1.
\end{equation}

The  behavior of $I_\nu(z)$ for large values of $|z|$ is given by the following asymptotic series first found by Hankel
\begin{equation}\label{ab}
I_\nu(z) = \frac{e^z}{(2\pi z)^{1/2}} \left[\sum_{k=0}^n \frac{(-1)^k (\nu,k)}{(2z)^k} + O(|z|^{-n-1})\right], \ \ \ \ \ |\arg z| \le \frac{\pi}2 - \delta,
\end{equation}
where we have denoted by $(\nu,k)$ the Hankel coefficients
\begin{equation}\label{nuk}
(\nu,k) = \frac{(-1)^k \G(1/2 - \nu + k) \G(1/2 + \nu + k)}{\G(k+1) \G(1/2-\nu) \G(1/2+\nu)},\ \ \ \  \ \ \text{so that}\ \ (\nu,0) = 1,
\end{equation}
see (5.11.10) on p. 123 in \cite{Le}, or 9.7.1 on p. 377 in \cite{AS}.
Note that \eqref{ab} implies in particular that for large $z>0$
\begin{equation}\label{ab2}
I_\nu(z) =  \frac{e^z}{(2\pi z)^{1/2}} \left(1+ O(|z|^{-1})\right).
\end{equation}
We also note the following Poisson representation of $I_\nu(z)$
\begin{equation}\label{prepInu}
I_\nu(z)  = \frac{(z/2)^\nu}{\sqrt{\pi} \G(\nu+1/2)} \int_{-1}^1 e^{zt} (1-t^2)^{\nu-1/2} dt,\ \ \ \ \ \ \ \ \ \Re \nu > - \frac 12,
\end{equation}
see  (10) on p. 81 in \cite{ET}.

For later use, we recall the  recurrence relations satisfied by $I_\nu(z)$, see e.g. (5.7.9) on p. 110 in \cite{Le}
\begin{equation}\label{rec}
\frac{d}{dz}[z^{\nu+1} I_{\nu+1}(z)] = z^{\nu+1} I_{\nu}(z),\ \ \ \ \ \ \ \frac{d}{dz}[z^{-\nu} I_{\nu}(z)] = z^{-\nu} I_{\nu+1}(z).
\end{equation}
For $z>0$ we have $I_\nu(z)>0$ for $\nu>-1$, and therefore we can rewrite \eqref{rec} as follows
\begin{equation}\label{rec2}
\frac{I_{\nu+1}'(z)}{I_\nu(z)} = 1 - \frac{\nu+1}z \frac{I_{\nu+1}(z)}{I_\nu(z)},\ \ \ \ \ \ \ \ \ \frac{I_{\nu}'(z)}{I_\nu(z)} = \frac{I_{\nu+1}(z)}{I_\nu(z)} + \frac{\nu}z,
\end{equation}

We next establish for the function $I_\nu$ a generalization of Weber's integral for the function $J_\nu$ (for the latter see (5.15.2) on p. 132 in \cite{Le}, or also formula 4. on p. 717 in \cite{GR}, or formula (10) on p. 29 in \cite{ET}). Such result is needed in the proof of Proposition \ref{P:sc} above.

\begin{lemma}\label{L:weber}
Let $\alpha >0$, $\Re \nu > -1$, then
\[
\int_0^\infty x^{\nu + 1} I_\nu(x) e^{-\alpha x^2} dx = 2^{-\nu-1} \alpha^{-\nu-1} e^{1/4\alpha}.
\]
\end{lemma}

\begin{proof}
We begin by observing that in view of \eqref{smallz} we have near $x = 0$
\[
|x^{\nu+1} I_\nu(x)| = O(|x|^{2\nu+1}),
\]
which is integrable since $\Re \nu>-1$. Furthermore, from \eqref{ab2} and $\alpha>0$ we see that the integrand is absolutely convergent near $\infty$. Thus the integrand belongs to $L^1(\R^+,dx)$. To compute the integral we apply \eqref{I} and integrate term by term obtaining
\begin{align*}
& \int_0^\infty x^{2\nu + 1} I_\nu(x) e^{-\alpha x^2} dx  = \sum_{k=0}^\infty \frac{1}{2^{\nu+2k} \G(k+1) \G(k+\nu+1)} \int_0^\infty x^{2\nu + 2k + 2} e^{-\alpha x^2} \frac{dx}{x} 
\\
& (\text{change of variable}\ y = \alpha x^2)
\\
& = \sum_{k=0}^\infty \frac{1}{2^{\nu+2k+1} \G(k+1) \G(k+\nu+1)} \int_0^\infty \left(\frac y\alpha\right)^{\nu + k + 1} e^{-y} \frac{dy}{y} 
\\
& = \sum_{k=0}^\infty \frac{1}{2^{\nu+2k+1} \G(k+1) \alpha^{k+\nu+1}} = \frac{1}{2^{\nu+1} \alpha^{\nu+1}} \sum_{k=0}^\infty \frac{1}{2^{2k} \G(k+1) \alpha^{k}}
\\
& = \frac{e^{1/4\alpha}}{2^{\nu+1} \alpha^{\nu+1}}.
\end{align*}

\end{proof}

For every $\nu>-1$ we next introduce the \emph{Bessel quotient} 
\begin{equation}\label{bq}
y_\nu(z) =  \frac{I_{\nu+1}(z)}{I_\nu(z)},\ \ \ \ \ \ \ \ \ \ z>0.
\end{equation}
From \eqref{smallz} we easily see that
\begin{equation}\label{zerozero}
y_\nu(z) \cong \frac{z}{2(\nu+1)},\ \ \ \ \text{as}\ z\to 0^+,
\end{equation}
and thus in particular
\begin{equation}\label{limitzero}
\underset{z\to 0^+}{\lim} y_\nu(z) = 0.
\end{equation}

Hankel's asymptotic formula \eqref{ab}, or also \eqref{ab2}, imply in particular that the line $y=1$ is a horizontal asymptote on $(0,\infty)$ for the function $z\to y_\nu(z)$, i.e., 
\begin{equation}\label{limitinfty}
\underset{z\to \infty}{\lim} y_\nu(z) = 1.
\end{equation}

Since as we have observed above one has $I_\nu(z)>0$ for every $z>0$ and $\nu>-1$, we infer that the function $z\to y_\nu(z)$ is continuous on $[0,\infty)$. Therefore, \eqref{limitinfty} and \eqref{limitzero} imply that for every $\nu>-1$ there exists $M_\nu\in (0,\infty)$ such that
\begin{equation}\label{boundedquotient}
\underset{z>0}{\sup}\ y_\nu(z) = M_\nu.
\end{equation}

The functions $I_\nu$ and $y_\nu$ are connected by the following result.

\begin{prop}\label{P:conn}
For every $\nu>-1$ and $z>0$ one has
\[
I_\nu(z) = \frac{1}{\G(\nu+1)} \left(\frac z2\right)^\nu \exp\left(\int_0^z y_\nu(t) dt\right).
\]
\end{prop}

\begin{proof}
If we take the logarithmic derivative of the function $I_\nu$, we find
\begin{equation}\label{logder}
\frac{d}{dz} \log I_\nu(z) = \frac{I_\nu'(z)}{I_\nu(z)} =  \frac{I_{\nu+1}(z)}{I_\nu(z)} + \frac{\nu}z = y_\nu(z)+ \frac{d}{dz} \log z^\nu,
\end{equation}
where in the second to the last equality we have used the second equation in \eqref{rec2} above. We rewrite \eqref{logder} in the following form
\begin{equation}\label{logder2}
\frac{d}{dt} \log \frac{I_\nu(t)}{t^\nu} = y_\nu(t), \ \ \ \ \ \ \ \ t>0,
\end{equation}
and integrate it between $\ve$ and $z$, obtaining
\[
\log \frac{I_\nu(z)}{z^\nu} - \log \frac{I_\nu(\ve)}{\ve^\nu} = \int_\ve^z y_\nu(t) dt.
\]
Taking the limit as $\ve\to 0^+$, using \eqref{smallz}, and exponentiating, we reach the desired conclusion. 

\end{proof}

\begin{lemma}\label{L:bq}
For every $\nu> - 1$ the Bessel quotient $y(z) = y_\nu(z)$ satisfies the following Cauchy problem for the Riccati equation
\begin{equation}\label{deBQ}
y'(z) = 1 - y(z)^2 - \frac{2\nu+1}{z} y(z),\ \ \ \ \ \ \ \ \ \ \ y(0) = 0.
\end{equation}
\end{lemma}

\begin{proof}
The initial condition $y(0) = 0$ is nothing but \eqref{limitzero}. The equation \eqref{deBQ} is derived in the following way. Recall the recurrence relations \eqref{rec2} above.
Since the rule for differentiating a quotient gives 
\[
y'(z) = \frac{I_{\nu+1}'(z)}{I_\nu(z)} - \frac{I_{\nu}'(z)}{I_\nu(z)} \frac{I_{\nu+1}(z)}{I_\nu(z)},
\]
substituting \eqref{rec2} into the latter expression we easily obtain the differential equation in \eqref{deBQ}.  

\end{proof}

One has the asymptotic formula on $(0,\infty)$.

\begin{prop}\label{P:asy}
Let $\nu>-1$. One has
\begin{equation}\label{asy}
\underset{z\to \infty}{\lim} z\left[1-y_\nu(z)\right] = \frac{2\nu+1}2.
\end{equation}
\end{prop}

\begin{proof}
To prove \eqref{asy} we use the asymptotic expansions \eqref{ab} and \eqref{ab2}, which give
\begin{align*}
1 - y_\nu(z) & = \frac{I_\nu(z) - I_{\nu+1}(z)}{I_\nu(z)}
\\
& = \frac{\frac{e^z}{(2\pi z)^{1/2}}}{\frac{e^z}{(2\pi z)^{1/2}} \left(1+ O(z^{-1})\right)} \left[1 - \frac{(\nu,1)}{(2z)} - 1 + \frac{(\nu+1,1)}{(2z)}+ O\left(\frac{1}{z^2}\right)\right]
\\
& = \frac{1}{1+ O(z^{-1})} \left(\frac{-(\nu,1) + (\nu+1,1)}{(2z)} + O\left(\frac{1}{z^2}\right)\right)
\\
& = \left[\frac{-(\nu,1) + (\nu+1,1)}{2z} + O\left(\frac{1}{z^2}\right)\right] \left(1+O\left(\frac 1z\right)\right).
\end{align*}
Using \eqref{nuk} we find
\begin{align*}
- (\nu,k) + (\nu+1,1) & = \frac{\G(1/2 - \nu + 1) \G(1/2 + \nu + 1)}{\G(2) \G(1/2-\nu) \G(1/2+\nu)} 
\\
& - \frac{\G(1/2 - (\nu+1) + 1) \G(1/2 + (\nu+1) + 1)}{\G(2) \G(1/2-(\nu+1)) \G(1/2+(\nu+1))}
\\
& = (1/2 - \nu)(1/2 + \nu) - ((1/2 - \nu)-1)((1/2 + \nu)+1)
\\
& = - (1/2 - \nu) + (1/2 + \nu) +1 = 2 \nu + 1. 
\end{align*}
The latter two formulas prove that as $z\to \infty$
\begin{equation}\label{asy000}
z (1 - y_\nu(z)) = \left(\frac{2 \nu + 1}2 + O(z^{-1})\right)(1 +  O(z^{-1})) = \frac{2 \nu + 1}2 + O(z^{-1}).
\end{equation}

\end{proof}

\begin{remark}\label{R:asy}
We emphasize that \eqref{asy} implies, in particular, that for large $z>0$ we must have
\begin{equation}\label{asy2}
\begin{cases}
1- y_\nu(z) < 0,\ \ \ \ \ \ \ \ \text{when}\ -1<\nu<-1/2,
\\
1- y_\nu(z) > 0,\ \ \ \ \ \ \ \ \text{when}\ -1/2<\nu.
\end{cases}
\end{equation}
The equation \eqref{asy2} shows that the function $y_\nu(z)$ approaches the asymptotic line $y=1$ from above when $-1<\nu<-1/2$, and from below when $\nu > -1/2$. 
\end{remark}

In fact, when $\nu\ge - 1/2$ we have the following global information. 

\begin{prop}\label{P:soni}
For every $\nu\ge -\frac 12$ and for every $z>0$ one has
\begin{equation}\label{quotients}
y_{\nu}(z) <  1.
\end{equation}
\end{prop}

\begin{proof}
This result was proved by Soni in \cite{So} when $\nu> - 1/2$. Since it is not easy to obtain Soni's original paper, we provide its short proof in what follows. By \eqref{prepInu} above, we have
\[
I_\nu(z)  = \frac{2 (z/2)^\nu}{\sqrt{\pi} \G(\nu+1/2)} \int_{0}^1 \cosh(zt) (1-t^2)^{\nu-1/2} dt,\ \ \ \ \ \ \ \ \ \Re \nu > - \frac 12.
\]
From this formula we obtain
\[
I_{\nu+1}(z)  = \frac{2^{-\nu} z^{\nu+1}}{\sqrt{\pi} (\nu + 1/2)\G(\nu+1/2)} \int_{0}^1 \cosh(zt) (1-t^2)^{\nu+1/2} dt,\ \ \ \ \ \ \ \ \ \Re \nu > - \frac 12.
\]
Integrating by parts in the integral in the right-hand side we find
\[
I_{\nu+1}(z)  =  \frac{2 (z/2)^{\nu}}{\sqrt{\pi} \G(\nu+1/2)} \int_{0}^1 t \sinh(zt) (1-t^2)^{\nu-1/2} dt,\ \ \ \ \ \ \ \ \ \Re \nu > - \frac 12.
\]
These formulas give for $\nu > - \frac 12$
\begin{align*}
& I_\nu(z)  - I_{\nu+1}(z)  = \frac{2 (z/2)^\nu}{\sqrt{\pi} \G(\nu+1/2)} \int_{0}^1 [\cosh(zt) - t \sinh(zt)] (1-t^2)^{\nu-1/2} dt
\\
& = \frac{2 (z/2)^\nu}{\sqrt{\pi} \G(\nu+1/2)} \int_{0}^1 [\coth(zt) - t] \sinh(zt) (1-t^2)^{\nu-1/2} dt.
\end{align*}
Since $t\to \sinh(zt) (1-t^2)^{\nu-1/2}$ is $>0$ on the interval $[0,1]$, except at $t = 0$, we infer that \eqref{quotients} will be true if for every $z>0$ and $0\le t \le 1$ one has $\coth(zt) - t \ge 0$. Now, $\p_t (\coth(zt) - t) = - \frac{z}{\sinh^2(zt)} -1 < 0$ for every $0<t\le 1$. Therefore, for every $z>0$ the function $h(t) = \coth(zt) - t$ is strictly decreasing on $(0,1)$. Since $h(1) = \coth(z) - 1>0$, we conclude that $h(t) >0$ for every $0\le t \le 1$. This proves Soni's inequality \eqref{quotients} when $\nu > - 1/2$. 

Later, Nasell observed in \cite{Na} that the statement is also true when $\nu = -1/2$. To see this not that from \eqref{half} above we find 
\begin{equation}\label{monohalf}
I_{-1/2}(z) - I_{1/2}(z) =  \left(\frac{2}{\pi z}\right)^{1/2} e^{-z}.
\end{equation}
Dividing by $I_{-1/2}(z)$ in \eqref{monohalf}, we find
\[
1 - \frac{I_{1/2}(z)}{I_{-1/2}(z)} = \left(\frac{2}{\pi z}\right)^{1/2} \frac{e^{-z}}{I_{-1/2}(z)} = \frac{2 e^{-z}}{e^{z} + e^{-z}} = \frac{2}{e^{2z} + 1},
\]
or equivalently
\begin{equation}\label{Ihalf}
\frac{I_{1/2}(z)}{I_{-1/2}(z)} = 1 - \frac{2}{e^{2z} + 1} < 1.
\end{equation}
Notice that in agreement with \eqref{asy} in Proposition \ref{P:asy}, we have
\[
\underset{z\to \infty}{\lim} z\left[1 - \frac{I_{1/2}(z)}{I_{-1/2}(z)}\right] = 0.
\]
Although \eqref{Ihalf} shows that for every $z>0$ we have
\begin{equation}\label{Ihalf2}
\left(\frac{I_{1/2}(z)}{I_{-1/2}(z)}\right)^2 - 1 < 0,
\end{equation}
the estimate does prove that the bound with $1$ is not optimal, since in fact the more precise formula
\begin{equation}\label{mpb}
\left(\frac{I_{1/2}(z)}{I_{-1/2}(z)}\right)^2 - 1 = - \frac{4}{e^{2z} + 1}\left(1- \frac{1}{e^{2z} + 1}\right)
\end{equation}
is available.

\end{proof}

\begin{remark}\label{R:soni}
We mention that the limitation $\nu \ge - 1/2$ in the global estimate in Proposition \ref{P:soni} is responsible for the limitation $a\ge 0$ in Theorems \ref{T:LYgen} and \ref{T:harnackbessel0} above.
\end{remark}

We take the occasion here to correct an oversight in the formula (A.3) in \cite{YK}, where the value of the limit \eqref{asy} is stated to be $2\nu+1$. We were in fact indirectly led to rediscovering Proposition \ref{P:asy}, after realizing that (A.3) in \cite{YK} cannot possibly be correct since it would lead to the contradictory conclusion that, when $-1<\nu<-\frac 12$, the function
\[
f(z) = f_\nu(z) = z^2 (y_\nu^2(z) - 1),
\]
be bounded on the whole line. While, if true, such result would be great in settling the question raised in Remark \ref{R:otherrange} above, unfortunately, the function $f(z)$ cannot possibly be bounded on $(0,\infty)$ since, as a consequence of \eqref{asy}, we know that 
\begin{equation}\label{asy00}
\underset{z\to \infty}{\lim} f(z) = \infty.
\end{equation}
To see \eqref{asy00}, we note that \eqref{limitinfty} gives
\[
\underset{z\to \infty}{\lim} (1 + y(z)) = 2.
\]
If we write
\[
f(z) = - z^2 (1- y(z))(1+ y(z)),
\]
then  \eqref{asy00} follows, since we have by \eqref{asy}
\[
\underset{z\to \infty}{\lim} \left\{- z (1- y(z))(1+ y(z))\right\} >0.
\]

We now have
\[
f'(z) = - 2z(1- y^2(z)) + 2 y(z) y'(z) z^2.
\]
Using the differential equation \eqref{deBQ}, we find
\begin{align*}
f'(z) & = - 2z\left(1- y^2(z)\right) + 2 y(z) z^2 \left(1 - y(z)^2 - \frac{2\nu+1}{z} y(z)\right).
\end{align*}
We rewrite this equation as follows
\begin{align*}
\frac{f'(z)}z & = - 2\left(1- y^2(z)\right) + 2 y(z) z \left(1 - y(z)^2\right) - 2(2\nu+1)  y(z)^2.
\end{align*}
By \eqref{limitinfty} we see that $- 2(1- y^2(z)) \to 0$ and that $- 2(2\nu+1)  y(z)^2 \to - 2(2\nu+1)$ as $z\to \infty$.
On the other hand, if instead of \eqref{asy} we had (A.3) in \cite{YK}, then such fact and \eqref{limitinfty} would  give $2 y(z) z \left(1 - y(z)^2\right) \to 4(2\nu+1)$ as $z\to \infty$. We would conclude that
\[
\underset{z\to \infty}{\lim} \frac{f'(z)}z = 4(2\nu+1)- 2(2\nu+1) = 2(2\nu+1)<0.
\]
This limit relation would imply, in particular, that $f(z)$ is decreasing in a neighborhood of infinity, and since $f(z)\ge 0$ at infinity by \eqref{asy}, we would reach the conclusion that $f(z)$ be bounded at infinity. But this is a contradiction with \eqref{asy00}.

Having said this, we note that the correct limiting relation \eqref{asy} above implies that, in fact, what one has is
\[
\underset{z\to \infty}{\lim} \frac{f'(z)}z = 0, 
\]
and this leads to no contradiction.

We close this section with the second critical property of the Bessel quotient $y_\nu$ of interest in this paper, namely, its monotonicity. It is certainly well-known to most experts of special functions but the only proof we could locate is embedded in the discussion in the Appendix of \cite{YK}. Since we need this result in the proof of Theorems \ref{T:struwe0} and \ref{T:poon0}, for the reader's convenience we provide its proof.

\begin{prop}\label{P:monotonicityy}
When $\nu \ge - 1/2$ the Bessel quotient $y_\nu$ strictly increases on $(0,\infty)$ from $y_\nu(0) = 0$ to its asymptotic value $y_\nu(\infty) = 1$. If instead $-1<\nu < - 1/2$, then $y_\nu$ first increases to its absolute maximum $>1$, and then it becomes strictly decreasing to its asymptotic value $y_\nu(\infty) = 1$. 
\end{prop}

\begin{proof}
When $\nu = - 1/2$ the desired conclusion follows from the explicit expression in \eqref{Ihalf} above. We thus assume $\nu \not= - 1/2$, or equivalently $2\nu + 1\not=0$. Differentiating \eqref{deBQ} we obtain
\begin{equation}\label{2deBQ}
y''(z) = \frac{2\nu+1}{z^2} y(z) - \left(\frac{2\nu+1}z + 2 y(z)\right) y'(z).
\end{equation}
At a stationary point $z_0>0$ we have $y'(z_0) = 0$, and thus \eqref{2deBQ} gives
\[
y''(z_0) = \frac{2\nu+1}{z_0^2} y(z_0).
\]
Since $y(z)>0$ for every $z>0$, it ensues that the sign of $y''(z_0)$ is the same as that of $2\nu+1$. Therefore, $y$ can only have strict local minima if $\nu>-1/2$, strict local maxima if $-1<\nu<-1/2$. 

On the other hand, the equation \eqref{deBQ} implies that $y'$ cannot change sign when $\nu>-1/2$. Otherwise, there
would be a point $z_0>0$ at which $y'(z_0) = 0$. From what has been said, $y$ must have a strict local minimum at $z_0$, and therefore $y'(z) \le 0$ in a left neighborhood of $z_0$, whereas $y'(z) \ge 0$ in a small right neighborhood of $z_0$. Since by \eqref{deBQ} we obtain
\[
y'(0) = \frac{1}{2\nu+2}>0,
\]
we infer that there must be a point $0<z_1<z_0$ where $y'(z_1) = 0$ and where $y$ attains a local maximum. Since this is impossible from what has been said above, we conclude that $y'$ cannot change sign when $\nu>-1/2$, and therefore $y' > 0$ for such values of $\nu$. When instead $-1<\nu<-1/2$, then \eqref{asy} above shows that in a neighborhood of infinity we must have $y>1$. Since $y(\infty) = 1$ there must be a point $z_0>0$ where $y'(z_0) < 0$. Since as we have observed $y'(0) >0$, there must be a point $z_1\in (0,z_0)$ where $y'(z_1) = 0$. From what we have said, $y$ can only have a strict local maximum at $z_1$. Since the equation \eqref{deBQ} gives
\[
0 = y'(z_1) = 1 - y(z_1)^2 - \frac{2\nu+1}{z_1} y(z_1),
\]
we conclude that $y(z_1)>1$. Since $y'$ can only change sign once, we finally infer that $y(z_1)$ is not only a strict local maximum, but also a global one.

\end{proof}

\bibliographystyle{amsplain}

\end{document}